\documentclass{amsart} 

\usepackage[english]{babel}
\usepackage[latin1]{inputenc}

\usepackage{amsmath}
\usepackage{amsthm}
\usepackage{amssymb}
\usepackage{tikz}

\usepackage{imakeidx} 
\usepackage{appendix}
\usepackage{tikz-cd}
\usepackage{verbatim}
\usepackage{enumitem}
\usepackage{multicol}

\usepackage[backref=page]{hyperref}
\usepackage{cleveref}

\hypersetup{colorlinks=true,linkcolor=blue,anchorcolor=blue,citecolor=blue}

\usepackage{adjustbox}

\usepackage{mathtools}

\newtheorem{thm}{Theorem}[section]
\newtheorem{prop}[thm]{Proposition}
\newtheorem{lemma}[thm]{Lemma}
\newtheorem{cor}[thm]{Corollary}

\newtheorem{question}[thm]{Question}

\theoremstyle{definition}

\newtheorem{claim}[thm]{Claim}

\theoremstyle{remark}

\numberwithin{equation}{section}

\newcommand{\Q}{\mathbb Q}
\newcommand{\F}{\mathbb F}

\newcommand{\Z}{\mathbb Z}
\newcommand{\G}{\mathbb G}

\renewcommand{\P}{\mathbb P}
\newcommand{\Spec}{\operatorname{Spec}}
\newcommand{\Br}{\operatorname{Br}}

\newcommand{\cl}{\overline}
\newcommand{\set}[1]{\left\{#1\right\}}
\renewcommand{\phi}{\varphi}

\newcommand{\on}[1]{\operatorname{#1}}
\newcommand{\ang}[1]{\langle{#1}\rangle}

\newcommand{\Prod}{\operatornamewithlimits{\textstyle\prod}}
\newcommand{\Sum}{\operatornamewithlimits{\textstyle\sum}}

\newcommand\smallbullet{
	\raisebox{-0.25ex}{\scalebox{1.2}{$\cdot$}}
}

\title{Non-formality of Galois cohomology modulo all primes}

\address{Department of Mathematics\\
	University of California\\
	Los Angeles, CA 90095\\ 
	United States of America}

\author{Alexander Merkurjev}
\email{merkurev@math.ucla.edu}

\author{Federico Scavia}
\email{scavia@math.ucla.edu}

\makeatletter
\@namedef{subjclassname@2020}{\textup{2020} Mathematics Subject Classification}
\makeatother

\date{September 2023}


\begin{document}

	\begin{abstract}
		Let $p$ be a prime number and let $F$ be a field of characteristic different from $p$. We prove that there exist a field extension $L/F$ and $a,b,c,d$ in $L^{\times}$ such that $(a,b)=(b,c)=(c,d)=0$ in $\on{Br}(F)[p]$ but $\langle a,b,c,d\rangle$ is not defined over $L$. Thus the Strong Massey Vanishing Conjecture at the prime $p$ fails for $L$, and the cochain differential graded ring $C^{\smallbullet}(\Gamma_L,\Z/p\Z)$ of the absolute Galois group $\Gamma_L$ of $L$ is not formal. This answers a question of Positselski.  
	\end{abstract}

	\maketitle
	
	\section{Introduction}
	Let $p$ be a prime number, let $F$ be a field of characteristic different from $p$ and containing a primitive $p$-th root of unity $\zeta$, and let $\Gamma_F$ be the absolute Galois group of $F$. The Norm-Residue Isomorphism Theorem of Voevodsky and Rost \cite{haesemeyer2019norm} gives an explicit presentation by generators and relations of the cohomology ring $H^{\smallbullet}(F,\Z/p\Z)=H^{\smallbullet}(\Gamma_F,\Z/p\Z)$. In view of this complete description of the cup product, the research on $H^{\smallbullet}(F, \Z/p\Z)$ shifted in recent years to external operations, defined in terms of the differential graded ring of continuous cochains $C^{\smallbullet}(\Gamma_F, \Z/p\Z)$. 
	
	Hopkins--Wickelgren \cite{hopkins2015splitting} asked whether $C^{\smallbullet}(\Gamma_F, \Z/p\Z)$ is formal for every field $F$ and every prime $p$. Loosely speaking, this amounts to saying that no essential information is lost when passing from $C^{\smallbullet}(\Gamma_F, \Z/p\Z)$ to $H^{\smallbullet}(F, \Z/p\Z)$. Positselski \cite{positselski2017koszulity} showed that $C^{\smallbullet}(\Gamma_F, \Z/p\Z)$ is not formal for some finite extensions $F$ of $\Q_{\ell}$ and $\F_{\ell}((z))$, where $\ell\neq p$. He then posed the following question; see \cite[p. 226]{positselski2017koszulity}.
	
	\begin{question}[Positselski]\label{non-formal-question}
		Does there exist a field $F$ containing all roots of unity of $p$-power order such that $C^{\smallbullet}(\Gamma_F,\Z/p\Z)$ is not formal?
	\end{question}
	
	We showed in \cite[Theorem 1.6]{merkurjev2022degenerate} that \Cref{non-formal-question} has a positive answer when $p=2$. In the present work we provide examples showing that the answer to \Cref{non-formal-question} is affirmative for all primes $p$.
	
	\begin{thm}\label{main-nonformal}
		Let $p$ be a prime number and let $F$ be a field of characteristic different from $p$. There exists a field $L$ containing $F$ such that the differential graded ring $C^{\smallbullet}(\Gamma_L,\Z/p\Z)$ is not formal.
	\end{thm}
	
	In order to detect non-formality of the cochain differential graded ring, we use Massey products. For any $n\geq 2$ and all $\chi_1,\dots,\chi_n\in H^1(F,\Z/p\Z)$, the Massey product of $\chi_1,\dots,\chi_n$ is a certain subset $\ang{\chi_1,\dots,\chi_n}\subset H^2(F,\Z/p\Z)$; see \Cref{massey-section} for the definition. We say that $\ang{\chi_1,\dots,\chi_n}$ is defined if it is not empty, and that it vanishes if it contains $0$. When $\on{char}(F)\neq p$ and $F$ contains a primitive $p$-th root of unity $\zeta$, Kummer Theory gives an identification $H^1(F,\Z/p\Z)=F^{\times}/F^{\times p}$, and we may thus consider Massey products $\ang{a_1,\dots,a_n}$, where $a_i\in F^\times$ for $1\leq i\leq n$.
	
	Let $n\geq 3$ be an integer, let $\chi_1,\dots,\chi_n\in H^1(F,\Z/p\Z)$, and consider the following assertions:
	\begin{align}
		\label{assertion-1} & \text{The Massey product $\ang{\chi_1,\dots,\chi_n}$ vanishes.} \\
		\label{assertion-2} & \text{The Massey product $\ang{\chi_1,\dots,\chi_n}$ is defined.} \\
		\label{assertion-3} & \text{We have $\chi_i\cup\chi_{i+1}=0$ for all $1\leq i\leq n-1$.}
	\end{align}
	We have that (\ref{assertion-1}) implies (\ref{assertion-2}), and that (\ref{assertion-2}) implies (\ref{assertion-3}). The Massey Vanishing Conjecture, due to Min\'{a}\v{c}--T\^{a}n \cite{minac2017triple} and inspired by the earlier work of Hopkins--Wickelgren \cite{hopkins2015splitting}, predicts that (\ref{assertion-2}) implies (\ref{assertion-1}). This conjecture has sparked a lot of activity in recent years. When $F$ is an arbitrary field, the conjecture is known when either $n=3$ and $p$ is arbitrary, by Efrat--Matzri and Min\'{a}\v{c}--T\^{a}n \cite{matzri2014triple, efrat2017triple, minac2016triple}, or $n=4$ and $p=2$, by \cite{merkurjev2023massey}. When $F$ is a number field, the conjecture was proved for all $n\geq 3$ and all primes $p$, by Harpaz--Wittenberg \cite{harpaz2019massey}. 
	
	When $n=3$, it is a direct consequence of the definition of Massey product that (\ref{assertion-3}) implies (\ref{assertion-2}). Thus (\ref{assertion-1}), (\ref{assertion-2}) and (\ref{assertion-3}) are equivalent when $n=3$.
	
	In \cite[Question 4.2]{minac2017counting}, Min\'{a}\v{c} and T\^{a}n asked whether (\ref{assertion-3}) implies (\ref{assertion-1}). This became known as the Strong Massey Vanishing Conjecture (see e.g. \cite{pal2018strong}): If $F$ is a field, $p$ is a prime number and $n\geq 3$ is an integer then, for all characters $\chi_1,\dots,\chi_n\in H^1(F,\Z/p\Z)$ such that $\chi_i\cup\chi_{i+1}=0$ for all $1\leq i\leq n-1$, the Massey product $\ang{\chi_1,\dots,\chi_n}$ vanishes.
	
	The Strong Massey Vanishing Conjecture implies the Massey Vanishing Conjecture. However, Harpaz and Wittenberg produced a counterexample to the Strong Massey Vanishing Conjecture, for $n=4$, $p=2$ and $F=\Q$; see \cite[Example A.15]{guillot2018fourfold}. More precisely, if we let $b=2$, $c=17$ and $a=d=bc=34$, then $(a,b)=(b,c)=(c,d)=0$ in $\on{Br}(\Q)$ but $\ang{a,b,c,d}$ is not defined over $\Q$. In this example, the classes of $a,b,c,d$ in $F^{\times}/F^{\times 2}$ are not $\F_2$-linearly independent modulo squares. In fact, by a theorem of Guillot--Min\'{a}\v{c}--Topaz--Wittenberg \cite{guillot2018fourfold}, if $F$ is a number field and $a,b,c,d$ are independent in $F^\times /F^{\times 2}$ and satisfy $(a,b)=(b,c)=(c,d)=0$ in $\on{Br}(F)$, then $\ang{a,b,c,d}$ vanishes. 
	
	If $F$ is a field for which the Strong Massey Vanishing Conjecture fails, for some $n\geq 3$ and some prime $p$, then $C^{\smallbullet}(\Gamma_F,\Z/p\Z)$ is not formal; see \Cref{formal-defined} for the $n=4$ case. Therefore \Cref{main-nonformal} follows from the next more precise result.
	
	\begin{thm}\label{main-explicit}
		Let $p$ be a prime number, let $F$ be a field of characteristic different from $p$. There exist a field $L$ containing $F$ and $\chi_1,\chi_2,\chi_3,\chi_4\in H^1(L,\Z/p\Z)$ such that $\chi_1\cup\chi_2=\chi_2\cup\chi_3=\chi_3\cup\chi_4=0$ in $H^2(L,\Z/p\Z)$ but $\ang{\chi_1,\chi_2,\chi_3,\chi_4}$ is not defined. Thus the Strong Massey Vanishing conjecture at $n=4$ and the prime $p$ fails for $L$, and $C^{\smallbullet}(\Gamma_L,\Z/p\Z)$ is not formal.
	\end{thm}
	This gives the first counterexamples to the Strong Massey Vanishing Conjecture for all odd primes $p$. We easily deduce that (\ref{assertion-3}) does not imply (\ref{assertion-2}) for all $n\geq 4$, in general: indeed, if the fourfold Massey product $\ang{\chi_1,\chi_2,\chi_3,\chi_4}$ is not defined, neither is the $n$-fold Massey product $\ang{\chi_1,\chi_2,\chi_3,\chi_4,0,\dots,0}$. \Cref{main-explicit} was proved in \cite[Theorem 1.6]{merkurjev2022degenerate} when $p=2$, and is new when $p$ is odd. Our proof of \Cref{main-explicit} is uniform in $p$.
	
	We now describe the main ideas that go into the proof of \Cref{main-explicit}. We may assume without loss of generality that $F$ contains a primitive $p$-th root of unity. In \Cref{section-preliminaries}, we collect preliminaries on formality, Massey products and Galois algebras. In particular, we recall Dwyer's Theorem (see \Cref{dwyer}): a Massey product $\ang{\chi_1,\dots,\chi_n}\subset H^2(F,\Z/p\Z)$ vanishes (resp. is defined) if and only if the homomorphism $(\chi_1,\dots,\chi_n)\colon \Gamma_F\to (\Z/p\Z)^n$ lifts to the group $U_{n+1}$ of upper unitriangular matrices in $\on{GL}_{n+1}(\F_p)$ (resp. to the group $\cl{U}_{n+1}$ of upper unitriangular matrices in $\on{GL}_{n+1}(\F_p)$ with top-right corner removed). As for \cite[Theorem 1.6]{merkurjev2022degenerate}, our approach is based on \Cref{dwyer-cor}, which is a restatement of \Cref{dwyer} in terms of Galois algebras.
	
	In \Cref{u5-bar-section}, we show that a fourfold Massey product $\ang{a,b,c,d}$ is defined over $F$ if and only if a certain system of equations admits a solution over $F$, and the variety defined by these equations is a torsor under a torus; see \Cref{uu5}. This is done by using Dwyer's \Cref{dwyer} to rephrase the property that $\ang{a,b,c,d}$ is defined in terms of $\cl{U}_5$-Galois algebras, and then by a detailed study of Galois $G$-algebras, for $G=U_3,\cl{U}_4,U_4,\cl{U}_5$. As a consequence, we also obtain an alternative proof of the Massey Vanishing Conjecture for $n=3$ and any prime $p$; see \Cref{mvc-3}. 
	
	In \Cref{generic-variety-section}, we use the work of \Cref{uu5-sec} to construct a ``generic variety'' for the property that $\ang{a,b,c,d}$ is defined. More precisely, under the assumption that $(a,b)=(c,d)=0$ in $\on{Br}(F)$ and letting $X$ be the Severi-Brauer variety of $(b,c)$, we construct an $F$-torus $T$, and a $T_{F(X)}$-torsor $E_w$ such that, if $E_w$ is non-trivial, then $\ang{a,b,c,d}$ is not defined over $F(X)$; see \Cref{generic-var}. The definition of $E_w$ depends on a rational function $w\in F(X)^\times$, which we construct in \Cref{construct-w}(3). 
	
	Since $(a,b)=(b,c)=(c,d)=0$ in $\on{Br}(F(X))$, the proof of \Cref{main-explicit} will be complete once we give an example of $a,b,c,d$ for which the corresponding torsor $E_w$ is non-trivial. Here we consider the generic quadruple $a,b,c,d$ such that $(a,b)$ and $(c,d)$ are trivial. More precisely, we let $x,y$ be two variables over $F$, and replace $F$ by $E\coloneqq F(x,y)$. We then set $a\coloneqq 1-x$, $b\coloneqq x$, $c\coloneqq y$ and $d\coloneqq 1-y$ over $E$. We have $(a,b)=(b,c)=0$ in $\Br(E)$. The class $(b,c)$ is not zero in $\Br(E)$, so that the Severi-Brauer variety $X/E$ of $(b,c)$ is non-trivial, but $(b,c)=0$ over $L\coloneqq E(X)$. 
	
	It is natural to attempt to prove that $E_w$ is non-trivial over $L$ by performing residue calculations to deduce that this torsor is ramified. However, the torsor $E_w$ is in fact unramified. We are thus led to consider a finer obstruction to the triviality of $E_w$. This ``secondary obstruction'' is only defined for unramified torsors. We describe the necessary homological algebra in \Cref{appendix-a}, and we define the obstruction and give a method to compute it in \Cref{appendix-b}. In \Cref{section-5}, an explicit calculation shows that the obstruction is non-zero on $E_w$, and hence $E_w$ is non-trivial, as desired.
	
	\subsection*{Notation}
	Let $F$ be a field, let $F_s$ be a separable closure of $F$, and denote by $\Gamma_F\coloneqq \on{Gal}(F_s/F)$ the absolute Galois group of $F$.
	
	If $E$ is an $F$-algebra, we let $H^i(E,-)$ be the \'etale cohomology of $\Spec(E)$ (possibly non-abelian if $i\leq 1$). If $E$ is a field, $H^i(E,-)$ may be identified with the continuous cohomology of $\Gamma_E$. 
	
	We fix a prime $p$, and we suppose that $\on{char}(F)\neq p$. If  $E$ is an $F$-algebra and $a_1,\dots,a_n\in E^{\times}$, we define the \'etale $E$-algebra $E_{a_1,\dots,a_n}$ by \[E_{a_1,\dots,a_n}\coloneqq E[x_1,\dots,x_n]/(x_1^p-a_1,\dots,x_n^p-a_n)\]
	and we set $(a_i)^{1/p}\coloneqq x_i$. More generally, for all integers $d$, we set $(a_i)^{d/p}\coloneqq x_i^d$. We denote by $R_{a_1,\dots,a_n}(-)$ the functor of Weil restriction along $F_{a_1,\dots,a_n}/F$. In particular $R_{a_1,\dots,a_n}(\G_{\on{m}})$ is the quasi-trivial torus associated to $F_{a_1,\dots,a_n}/F$, and we denote by $R^{(1)}_{a_1,\dots,a_n}(\G_{\on{m}})$ the norm-one subtorus of $R_{a_1,\dots,a_n}(\G_{\on{m}})$. We denote by $N_{a_1,\dots,a_n}$ the norm map from $F_{a_1,\dots,a_n}$ to $F$.
	
	We write $\on{Br}(F)$ for the Brauer group of $F$. If $\on{char}(F)\neq p$ and $F$ contains a primitive $p$-th root of unity, for all $a,b\in F^\times$ we let $(a,b)$ be the corresponding degree-$p$ cyclic algebra and for its class in $\on{Br}(F)$; see \Cref{kummer-sub}. We denote by $N_{a_1,\dots,a_n}\colon\on{Br}(F_{a_1,\dots,a_n})\to \on{Br}(F)$ for the corestriction map along $F_{a_1,\dots,a_n}/F$. 
	
	An $F$-variety is a separated integral $F$-scheme of finite type. If $X$ is an $F$-variety, we denote by $F(X)$ the function field of $X$, and we write $X^{(1)}$ for the collection of all points of codimension $1$ in $X$. We set $X_s\coloneqq X\times_FF_s$. If $K$ is an \'etale algebra over $F$, we write $X_K$ for $X\times_FK$. For all $a_1,\dots,a_n\in F^\times$, we write $X_{a_1,\dots,a_n}$ for $X_{F_{a_1,\dots,a_n}}$. When $X=\P^d_F$ is a $d$-dimensional projective space, we denote by $\P^d_{a_1,\dots,a_n}$ the base change of $\P^d_F$ to $F_{a_1,\dots,a_d}$.

	\section{Preliminaries}\label{section-preliminaries}

	\subsection{Galois algebras and Kummer Theory}\label{kummer-sub}
	
	Let $F$ be a field and let $G$ be a finite group. A $G$\emph{-algebra} is an \'etale $F$-algebra $L$ on which $G$ acts via $F$-algebra automorphisms. The $G$-algebra $L$ is \emph {Galois} if $|G|=\dim_F(L)$ and $L^G=F$; see \cite[Definitions (18.15)]{knus1998book}. A $G$-algebra $L/F$ is Galois if and only if the morphism of schemes $\on{Spec}(L)\to \on{Spec}(F)$ is an \'etale $G$-torsor. If $L/F$ is a Galois $G$-algebra, the group algebra $\Z[G]$ acts on the multiplicative group $L^{\times}$: an element $\Sum_{i=1}^r m_ig_i\in \Z[G]$, where $m_i\in \Z$ and $g_i\in G$, sends $x\in L^{\times}$ to $\Prod_{i=1}^r g_i(x)^{m_i}$. 
	
	By  \cite[Example (28.15)]{knus1998book}, we have a canonical bijection 
	\begin{equation}\label{galois-alg}
		\on{Hom}_{\on{cont}}(\Gamma_F,G)/_{\sim}\xrightarrow{\sim}\set{\text{Isomorphism classes of Galois $G$-algebras over $F$}},
	\end{equation}
	where, if $f_1,f_2\colon \Gamma_F\to G$ are continuous group homomorphisms, we say that $f_1\simeq f_2$ if there exists $g\in G$ such that $gf_1(\sigma)g^{-1}=f_2(\sigma)$ for all $\sigma\in \Gamma_F$.
	
	Let $H$ be a normal subgroup of $G$. Under the correspondence (\ref{galois-alg}), the map $\on{Hom}_{\on{cont}}(\Gamma_F,G)/_{\sim}\to \on{Hom}_{\on{cont}}(\Gamma_F,G/H)/_{\sim}$ sends the class of a Galois $G$-algebra $L$ to the class of the Galois $G/H$-algebra $L^H$.
	
	\begin{lemma}\label{pushout-g-algebras}
		Let $G$ be a finite group, and let $H, H', S$ be normal subgroups of $G$ such that $H\subset S$, $H'\subset S$, and the square
		\begin{equation}\label{g-h-h'-s}
			\begin{tikzcd}
				G \arrow[r] \arrow[d] & G/H \arrow[d] \\
				G/H' \arrow[r] & G/S 
			\end{tikzcd}
		\end{equation}
		is cartesian. 
		
		(1) Let $L$ be a Galois $G$-algebra. Then the tensor product $L^H \otimes_{L^S} L^{H'}$ has a Galois $G$-algebra structure given by $g(x\otimes x')\coloneqq g(x)\otimes g(x')$ for all $x\in L^H$ and $x'\in L^{H'}$. Moreover, the inclusions $L^H\to L$ and $L^{H'}\to L$ induce an isomorphism of Galois $G$-algebras $L^H \otimes_{L^S} L^{H'} \rightarrow L$.
		
		(2) Conversely, let $K$ be a Galois $G/H$-algebra, let $K'$ be a Galois $G/H'$-algebra, and let $E$ be a Galois $G/S$-algebra. Suppose given $G$-equivariant algebra homomorphisms $E \rightarrow K$ and $E \rightarrow K'$. Endow the tensor product $L \coloneqq K \otimes_{E} K'$ with the structure of a $G$-algebra given by $g(x\otimes x')\coloneqq g(x)\otimes g(x')$ for all $x\in K$ and $x'\in K'$. Then $L$ is a Galois $G$-algebra such that $L^H\simeq K$ as $G/H$-algebras, and $L^{H'}\simeq K'$ as $G/H'$-algebras.
	\end{lemma} 
	The condition that (\ref{g-h-h'-s}) is cartesian is equivalent to $H\cap H'= \set{1}$ and $S=HH'$.
	
	\begin{proof}
		(1) It is clear that the formula $g(x\otimes x')\coloneqq g(x)\otimes g(x')$ makes $L^H \otimes_{L^S} L^{H'}$ into a $G$-algebra. Consider the commutative square of $F$-schemes
		\[
		\begin{tikzcd}
			\Spec(L) \arrow[r] \arrow[d] & \Spec(L)/H' \arrow[d] \\
			\Spec(L)/H \arrow[r] & \Spec(L)/S.
		\end{tikzcd}
		\]
		After base change to a separable closure of $F$, this square becomes the cartesian square (\ref{g-h-h'-s}), and therefore it is cartesian. Passing to coordinate rings, we deduce that the map $L^H \otimes_{L^S} L^{H'} \rightarrow L$ is an isomorphism of $G$-algebras. In particular, since $L$ is a Galois $G$-algebra, so is $L^H \otimes_{L^S} L^{H'}$. 
		
		(2) We have a $G$-equivariant cartesian diagram
		\[
		\begin{tikzcd}
			\Spec(L) \arrow[r]\arrow[d]  & \Spec(K')\arrow[d] \\
			\Spec(K) \arrow[r] & \Spec(E).
		\end{tikzcd}
		\]
		Every $G$-equivariant morphism between $G/H$ and $G/S$ is isomorphic to the projection map $G/H\to G/S$. Therefore the base change of $\Spec(K) \to \Spec(E)$ to $F_s$ is $G$-equivariantly isomorphic to the projection $G/H \to G/S$. Similarly for $\Spec(K')\to\Spec(E)$. Therefore the base change of $\Spec(L)\to \Spec(F)$ over $F_s$ is $G$-equivariantly isomorphic to $(G/H)\times_{G/S}(G/H')\simeq G$, that is, the morphism $\Spec(L)\to \Spec(F)$ is an \'etale $G$-torsor. 
	\end{proof}
	
	Suppose that $\on{char}(F)\neq p$ and that $F$ contains a primitive $p$-th root of unity. We fix a primitive $p$-th root of unity $\zeta\in F^\times$. This determines an isomorphism of Galois modules $\Z/p\Z \simeq \mu_p$, given by $1\mapsto\zeta$, and so the Kummer sequence yields an isomorphism
	\begin{equation}\label{kummer}\on{Hom}_{\on{cont}}(\Gamma_F,\Z/p\Z)=H^1(F,\Z/p\Z )\simeq H^1(F,\mu_p)\simeq F^{\times}/F^{\times p}.\end{equation}
	
	For every $a\in F^{\times}$, we let $\chi_a\colon \Gamma_F\to \Z/p\Z$ be the homomorphism corresponding to the coset $a F^{\times p}$ under (\ref{kummer}). Explicitly, letting $a'\in F_{\on{sep}}^\times$ be such that $(a')^p=a$, we have $g(a')=\zeta^{\chi_a(g)}a'$ for all $g\in \Gamma_F$. This definition does not depend on the choice of $a'$.
	
	Now let $n\geq 1$ be an integer. For all $i=1,\dots,n$, let $\sigma_i$ be the canonical generator of the $i$-th factor $\Z/p\Z$ of $(\Z/p\Z)^n$. By (\ref{kummer}) all Galois $(\Z/p\Z)^n$-algebras over $F$ are of the form $F_{a_1,\dots,a_n}$, where $a_1,\dots,a_n\in F^\times$ and the Galois $(\Z/p\Z)^n$-algebra structure is defined by $(\sigma_i-1)a_i^{1/p}=\zeta$ for all $i$ and $(\sigma_i-1)a_j^{1/p}=1$ for all $j\neq i$.

	We write $(a,b)$ for the cyclic degree-$p$ central simple algebra over $F$ generated, as an $F$-algebra, by $F_a$ and an element $y$ such that
	\[y^p=b,\qquad ty=y\sigma_a(t)\ \text{for all $t\in F_a$}.\]
	We also write $(a,b)$ for the class of $(a,b)$ in $\on{Br}(F)$.
	The Kummer sequence yields a group isomorphism \[\iota\colon H^2(F,\Z/p\Z)\xrightarrow{\sim}\on{Br}(F)[p].\] 
	For all $a,b\in F^{\times}$, we have $\iota(\chi_a\cup \chi_b)=(a,b)$ in $\on{Br}(F)$; see \cite[Chapter XIV, Proposition 5]{serre1979local}. 
	
	\begin{lemma}\label{cup-norm}
		Let $p$ be a prime, and let $F$ be a field of characteristic different from $p$ and containing a primitive $p$-th root of unity $\zeta$. The following are equivalent:
		
		(i) $(a,b)=0$ in $\on{Br}(F)$;
		
		(ii) there exists $\alpha\in F_a^\times$ such that $b=N_a(\alpha)$;
		
		(iii) there exists $\beta\in F_b^\times$ such that $a=N_b(\beta)$.
	\end{lemma}
	
	\begin{proof}
		See \cite[Chapter XIV, Proposition 4(iii)]{serre1979local}.
	\end{proof}

	\subsection{Formality and Massey products}\label{massey-section}
	
	Let $(A,\partial)$ be a differential graded ring, i.e, $A=\oplus_{i\geq 0}A^i$ is a non-negatively graded abelian group with an associative multiplication which respects the grading, and $\partial\colon A\to A$ is a group homomorphism of degree $1$ such that $\partial\circ \partial=0$ and $\partial(ab)=\partial(a)b+(-1)^ia\partial(b)$ for all $i\geq 0$, $a\in A^i$ and $b\in A$. We denote by $H^{\smallbullet}(A)\coloneqq \on{Ker}(\partial)/\on{Im}(\partial)$ the cohomology of $(A,\partial)$, and we write $\cup$ for the multiplication (cup product) on $H^{\smallbullet}(A)$.
	
	We say that $A$ is \emph{formal} if it is quasi-isomorphic, as a differential graded ring, to $H^{\smallbullet}(A)$ with the zero differential.

	Let $n\geq 2$ be an integer and $a_1,\dots,a_n\in H^1(A)$. A \emph{defining system} for the $n$-th order Massey product $\ang{a_1,\dots,a_n}$ is a collection $M$ of elements of $a_{ij}\in A^1$, where $1\leq i<j\leq n+1$, $(i,j)\neq (1,n+1)$, such that
	\begin{enumerate}
		\item $\partial(a_{i,i+1})=0$ and $a_{i,i+1}$ represents $a_i$ in $H^1(A)$, and
		\item $\partial(a_{ij})=-\Sum_{l=i+1}^{j-1}a_{il}a_{lj}$ for all $i<j-1$.
	\end{enumerate}
	
	It follows from (2) that  $-\Sum_{l=2}^na_{1l}a_{l,n+1}$ is a $2$-cocycle: we write $\ang{a_1,\dots,a_n}_M$ for its cohomology class in $H^2(A)$, called the \emph{value} of $\ang{a_1,\dots,a_n}$ corresponding to $M$. By definition, the \emph{Massey product} of $a_1,\dots,a_n$ is the subset $\ang{a_1,\dots,a_n}$ of $H^2(A)$ consisting of the values $\ang{a_1,\dots,a_n}_M$ of all defining systems $M$. We say that the Massey product $\ang{a_1,\dots,a_n}$ is \emph{defined} if it is non-empty, and that it \emph{vanishes} if $0\in \ang{a_1,\dots,a_n}$.
	
	\begin{lemma}\label{formal-defined}
		Let $(A,\partial)$ be a differential graded ring, and let $\alpha_1,\alpha_2,\alpha_3,\alpha_4$ be elements of $H^1(A)$ satisfying $\alpha_1\cup\alpha_2=\alpha_2\cup\alpha_3=\alpha_3\cup\alpha_4=0$. If $A$ is formal, then $\ang{\alpha_1,\alpha_2,\alpha_3,\alpha_4}$ is defined.
	\end{lemma}
	
	\begin{proof}
		This was proved in \cite[Lemma B.1]{merkurjev2022degenerate} under the assumption that $A$ is a differential graded $\F_2$-algebra. The proof for an arbitrary differential graded ring remains the same.
	\end{proof}
	In fact, one could prove the following: If the differential graded ring $A$ is formal, then for all $n\geq 3$ and all $\alpha_1,\dots,\alpha_n\in H^1(A)$ such that $\alpha_i\cup\alpha_{i+1}=0$ for all $1\leq i\leq n-1$, then $\ang{\alpha_1,\dots,\alpha_n}$ vanishes.
	
	\subsection{Dwyer's Theorem}\label{dwyer-section}
	Let $p$ be a prime, and let $U_{n+1}\subset \on{GL}_{n+1}(\F_p)$ be the subgroup of $(n+1)\times (n+1)$ upper unitriangular matrices. For all $1\leq i<j\leq n+1$, we denote by $e_{ij}$ the matrix whose non-diagonal entries are all zero except for the entry $(i,j)$, which is equal to $1$. We set $\sigma_i\coloneqq e_{i,i+1}$ for all $1\leq i\leq n$. By \cite[Theorem 1]{biss2001presentation}, the group $U_{n+1}$ admits a presentation with generators the $\sigma_i$ and relations:
	\begin{equation}\label{relation0}
		\sigma_i^p=1\qquad\text{for all $1\leq i\leq n$,}
	\end{equation}
	\begin{equation}\label{relation1}
		[\sigma_i,\sigma_j]=1\qquad\text{for all $1\leq i\leq j-2\leq n-2$,}
	\end{equation}
	\begin{equation}\label{relation2}
		[\sigma_i,[\sigma_i,\sigma_{i+1}]]=[\sigma_{i+1},[\sigma_i,\sigma_{i+1}]]\qquad\text{for all $1\leq i\leq n-2$,}
	\end{equation}
	\begin{equation}\label{relation3}
		[[\sigma_i,\sigma_{i+1}],[\sigma_{i+1},\sigma_{i+2}]]=1\qquad\text{for all $1\leq i\leq n-3$.}
	\end{equation}

	The following relation holds in $U_{n+1}$:
	\[
	[e_{ij},e_{jk}]=e_{ik}\qquad\text{for all $1\leq i< j<k\leq n+1$.}
	\]
	By induction, we deduce that
	\[e_{1,n+1}=[\sigma_1,[\sigma_2,\dots,[\sigma_{n-2},[\sigma_{n-1},\sigma_n]]\dots]].\]
	
	The center $Z_{n+1}$ of $U_{n+1}$ is the subgroup generated by $e_{1,n+1}$. The factor group $\cl{U}_{n+1}\coloneqq U_{n+1}/Z_{n+1}$ may be identified with the group of all $(n+1)\times (n+1)$ upper unitriangular matrices with entry $(1,n+1)$ omitted. For all $1\leq i<j\leq n+1$, let $\cl{e}_{ij}$ be the coset of $e_{ij}$ in $\cl{U}_{n+1}$, and set $\cl{\sigma}_i\coloneqq \cl{e}_{i,i+1}$ for all $1\leq i\leq n$. Then $\cl{U}_{n+1}$ is generated by all the $\cl{e}_{ij}$ modulo the relations
	\begin{equation}\label{relation0-bar}
		\cl{\sigma}_i^p=1\qquad\text{for all $1\leq i\leq n$,}
	\end{equation}
	\begin{equation}\label{relation1-bar}
		[\cl{\sigma}_i,\cl{\sigma}_j]=1\qquad\text{for all $1\leq i\leq j-2\leq n-2$,}
	\end{equation}
	\begin{equation}\label{relation2-bar}
		[\cl{\sigma}_i,[\cl{\sigma}_i,\cl{\sigma}_{i+1}]=[\cl{\sigma}_{i+1},[\cl{\sigma}_i,\cl{\sigma}_{i+1}]]\qquad\text{for all $1\leq i\leq n-2$,}
	\end{equation}
	\begin{equation}\label{relation3-bar}
		[[\cl{\sigma}_i,\cl{\sigma}_{i+1}],[\cl{\sigma}_{i+1},\cl{\sigma}_{i+2}]]=1\qquad\text{for all $1\leq i\leq n-3$.}
	\end{equation}
	\begin{equation}\label{relation4-bar}
		[\cl{\sigma}_1,[\cl{\sigma}_2,\dots,[\cl{\sigma}_{n-2},[\cl{\sigma}_{n-1},\cl{\sigma}_n]]\dots]]=1.
	\end{equation}
	
	We write $u_{ij}\colon U_{n+1}\to \Z/p\Z$ for the $(i,j)$-th coordinate function on $U_{n+1}$. Note that $u_{ij}$ is not a group homomorphism unless $j=i+1$. We have commutative diagram
	\begin{equation}\label{central-exact-seq}
		\begin{tikzcd}
			1 \arrow[r] & Z_{n+1}\arrow[r]\arrow[r] & U_{n+1}\arrow[r]\arrow[dr] & \cl{U}_{n+1}\arrow[r]\arrow[d] & 1\\
			&&& (\Z/p\Z)^n
		\end{tikzcd}
	\end{equation}
	where the row is a central exact sequence and the homomorphism $U_{n+1}\to (\Z/p\Z)^n$ is given by $(u_{12},u_{23},\dots, u_{n,n+1})$. We also let 
	\[Q_{n+1}\coloneqq \on{Ker}[U_{n+1}\to (\Z/p\Z)^n],\quad \cl{Q}_{n+1}\coloneqq \on{Ker}[\cl{U}_{n+1}\to (\Z/p\Z)^n]=Q_{n+1}/Z_{n+1}.\]
	Note that $Z_{n+1}\subset Q_{n+1}$, with equality when $n=2$. 
	
	Let $G$ be a profinite group. The complex $(C^{\smallbullet}(G,\Z/p\Z),\partial)$ of mod $p$ non-ho\-mo\-ge\-neous continuous cochains of $G$ with the standard cup product is a differential graded ring. Therefore $H^{\smallbullet}(G,\Z/p\Z)=H^{\smallbullet}(C^{\smallbullet}(G,\Z/p\Z),\partial)$ is endowed with Massey products. The following theorem is due to Dwyer \cite{dwyer1975homology}.
	
	\begin{thm}[Dwyer]\label{dwyer}
		Let $p$ be a prime number, let $G$ be a profinite group, let $\chi_1,\dots,\chi_n\in H^1(G,\Z/p\Z)$, and write $\chi\colon G\to (\Z/p\Z)^n$ for the continuous homomorphism with components $(\chi_1,\dots,\chi_n)$. Consider (\ref{central-exact-seq}).
		
		(1) The Massey product $\ang{\chi_1,\dots,\chi_n}$ is defined if and only if $\chi$ lifts to a continuous homomorphism $G\to \cl{U}_{n+1}$. 
		
		(2) The Massey product $\ang{\chi_1,\dots,\chi_n}$ vanishes if and only if $\chi$ lifts to a continuous homomorphism $G\to U_{n+1}$. 
	\end{thm}
	
	\begin{proof}
		See \cite{dwyer1975homology} for Dwyer's original proof in the setting of abstract groups, and \cite{efrat2014zassenhaus} or \cite[Proposition 2.2]{harpaz2019massey} for the statement in the case of profinite groups.
	\end{proof}

	\Cref{dwyer} may be rephrased as follows.
	
	\begin{cor}\label{dwyer-cor}
		Let $p$ be a prime, $F$ be a field of characteristic different from $p$ and containing a primitive $p$-th root of unity $\zeta$, and let $a_1,\dots,a_n\in F^\times$. The Massey product $\ang{a_1,\dots,a_n}\subset H^2(F,\Z/p\Z)$ is defined (resp. vanishes) if and only if there exists a Galois $\cl{U}_{n+1}$-algebra $K/F$ (resp. a Galois $U_{n+1}$-algebra $L/F$)  such that $K^{\cl{Q}_{n+1}}\simeq F_{a_1,\dots,a_n}$ (resp. $L^{{Q}_{n+1}}\simeq F_{a_1,\dots,a_n}$) as $(\Z/p\Z)^n$-algebras.
	\end{cor}
	
	\begin{proof}
		This follows from \Cref{dwyer} and (\ref{galois-alg}).
	\end{proof}
	
	We will apply \Cref{pushout-g-algebras} to the cartesian square of groups
	\begin{equation}\label{phi-phi'}
		\begin{tikzcd}
			\cl{U}_{n+1} \arrow[d,"\varphi'_{n+1}"] \arrow[r,"\varphi_{n+1}"] & U_n \arrow[d,"\varphi'_n"] \\
			U_n \arrow[r,"\varphi_n"] & U_{n-1}	
		\end{tikzcd}
	\end{equation}
	where $\phi_{n+1}$ (respectively, $\phi'_{n+1}$) is the restriction homomorphism from $U_{n+1}$ or from $U_{n+1}$ to the top-left (respectively, bottom-right) $n\times n$ subsquare $U_n$ in $U_{n+1}$.
	
	The fact that the square (\ref{phi-phi'}) is cartesian is proved in \cite[Proposition 2.7]{merkurjev2022degenerate} when $p=2$. The proof extends to odd $p$ without change.
	
	\section{Massey products and Galois algebras}\label{u5-bar-section}
	
	In this section, we let $p$ be a prime number and we let $F$ be a field. With the exception of \Cref{mvc-3}, we assume that $\on{char}(F)\neq p$ and that $F$ contains a primitive $p$-th root of unity $\zeta$.
	
	\subsection{Galois \texorpdfstring{$U_3$}{U3}-algebras}
	
	Let $a,b\in F^\times$, and suppose that $(a,b)=0$ in $\on{Br}(F)$. By \Cref{cup-norm}, we may fix $\alpha\in F_a^\times$ and $\beta\in F_b^\times$ such that $N_a(\alpha)=b$ and $N_b(\beta)=a$. 
	
	We write $(\Z/p\Z)^2=\ang{\sigma_a,\sigma_b}$, and we view $F_{a,b}$ as a Galois $(\Z/p\Z)^2$-algebra as in \Cref{kummer-sub}. The projection $U_3\to \cl{U}_3=(\Z/p\Z)^2$ sends $e_{12}\mapsto \sigma_a$ and $e_{23}\mapsto \sigma_b$. We define the following elements of $U_3$:
	\[
	\sigma_a\coloneqq e_{12},\qquad
	\sigma_b\coloneqq e_{23},\qquad
	\tau\coloneqq e_{13}=[\sigma_a,\sigma_b].
	\]
	Suppose given $x\in F_a^\times$ such that 
	\begin{equation}\label{u3-x}(\sigma_a-1)x=\frac{b}{\alpha^p}. 
	\end{equation}
	The \'etale $F$-algebra $K\coloneqq (F_{a,b})_x$ has the structure of a Galois $U_3$-algebra such that the Galois $(\Z/p\Z)^2$-algebra $K^{Q_3}$ is equal to $F_{a,b}$, and
	\begin{equation}\label{u3-algebra-1}(\sigma_a-1)x^{1/p}=\frac{b^{1/p}}{\alpha},\qquad (\sigma_b-1)x^{1/p}=1,\qquad (\tau-1)x^{1/p}=\zeta^{-1}.\end{equation}
	Similarly, suppose given $y\in F_b^\times$ such that 
	\begin{equation}\label{u3-y}(\sigma_b-1)y=\frac{a}{\beta^p}.
	\end{equation}
	The \'etale $F$-algebra $K\coloneqq (F_{a,b})_y$ has the structure of a Galois $U_3$-algebra, such that the Galois $(\Z/p\Z)^2$-algebra $K^{Q_3}$ is equal to $F_{a,b}$, and
	\begin{equation}\label{u3-algebra-2}(\sigma_a-1)y^{1/p}=1,\qquad (\sigma_b-1)y^{1/p}=\frac{a^{1/p}}{\beta},\qquad (\tau-1)y^{1/p}=\zeta.
	\end{equation}
	In (\ref{u3-algebra-1}) and (\ref{u3-algebra-2}), the relation involving $\tau$ follows from the first two.
	
	If $x\in F_a^\times$ satisfies (\ref{u3-x}), then so does $ax$. We may thus apply (\ref{u3-algebra-1}) to $(F_{a,b})_{ax}$. Therefore $(F_{a,b})_{ax}$ has the structure of a Galois $U_3$-algebra, where $U_3$ acts via $\cl{U}_3=\on{Gal}(F_{a,b}/F)$ on $F_{a,b}$, and
	\[(\sigma_a-1)(ax)^{1/p}=\frac{b^{1/p}}{\alpha},\qquad (\sigma_b-1)(ax)^{1/p}=1,\qquad (\tau-1)(ax)^{1/p}=\zeta^{-1}.\]
	Similarly, if $y\in F_b^\times$ satisfies (\ref{u3-y}), we may apply (\ref{u3-algebra-2}) to $(F_{a,b})_{by}$. Therefore $(F_{a,b})_{by}$ admits a Galois $U_3$-algebra structure, where $U_3$ acts via $\cl{U}_3=\on{Gal}(F_{a,b}/F)$ on $F_{a,b}$, and 
	\[(\sigma_a-1)(by)^{1/p}=1,\qquad (\sigma_b-1)(by)^{1/p}=\frac{a^{1/p}}{\beta},\qquad (\tau-1)(by)^{1/p}=\zeta.\]
	
	\begin{lemma}\label{minus-x}
		(1) Let $x\in F_a^\times$ satisfy (\ref{u3-x}), and consider the Galois $U_3$-algebras $(F_{a,b})_x$ and $(F_{a,b})_{ax}$ as in (\ref{u3-algebra-1}). Then $(F_{a,b})_x\simeq (F_{a,b})_{ax}$ as Galois $U_3$-algebras. 
		
		(2) Let $y\in F_b^\times$ satisfy (\ref{u3-x}), and consider the Galois $U_3$-algebras $(F_{a,b})_y$ and $(F_{a,b})_{by}$ as in (\ref{u3-algebra-2}). Then $(F_{a,b})_y\simeq (F_{a,b})_{by}$ as Galois $U_3$-algebras.
	\end{lemma}
	
	\begin{proof}
		(1) The automorphism $\sigma_b\colon F_{a,b}\to F_{a,b}$ extends to an isomorphism of \'etale algebras $f\colon (F_{a,b})_x\to (F_{a,b})_{ax}$ by sending $x^{1/p}$ to $(ax)^{1/p}a^{-1/p}$. The map $f$ is well defined because 
		$f(x^{1/p})^p=x=[(ax)^{1/p}a^{-1/p}]^p$. We check that it is $U_3$-equivariant. This is true on $F_{a,b}$ because  $\sigma_a\sigma_b=\sigma_b\sigma_a$ on $F_{a,b}$. Moreover,
		\begin{align*}
			\sigma_a(f(x^{1/p}))=\sigma_a((ax)^{1/p})\cdot \sigma_a(a^{-1/p})=(b^{1/p}/\alpha)(ax)^{1/p}\cdot\zeta a^{-1/p} \\
			=(\zeta b^{1/p}/\alpha)\cdot (ax)^{1/p}a^{-1/p}=f((b^{1/p}/\alpha)(x^{1/p}))=f(\sigma_a(x^{1/p}))
		\end{align*}
		and
		\[\sigma_b(f(x^{1/p}))=\sigma_b((ax)^{1/p})\cdot \sigma_b(a^{-1/p})=(ax)^{1/p}a^{-1/p}=f(x^{1/p})=f(\sigma_b(x^{1/p})).\]
		Thus $f$ is $U_3$-equivariant, as desired.
		
		(2) The proof is similar to that of (1).
	\end{proof}
	
	\begin{prop}\label{uu3}
		Let $a,b\in F^\times$ be such that $(a,b)=0$ in $\on{Br}(F)$, and fix $\alpha\in F_a^\times$ and $\beta\in F_b^\times$ such that $N_a(\alpha)=b$ and $N_b(\beta)=a$. 
		
		(1) Every Galois $U_3$-algebra $K$ over $F$ such that $K^{Q_3}\simeq F_{a,b}$ as $(\Z/p\Z)^2$-algebras is of the form $(F_{a,b})_x$ for some $x\in F_a^\times$ as in (\ref{u3-x}), with $U_3$-action given by (\ref{u3-algebra-1}).
		
		(2) Every Galois $U_3$-algebra $K$ over $F$ such that $K^{Q_3}\simeq F_{a,b}$ as $(\Z/p\Z)^2$-algebras is of the form $(F_{a,b})_y$ for some $y\in F_b^\times$ as in (\ref{u3-y}), with $U_3$-action given by (\ref{u3-algebra-2}).
		
		(3) Let $(F_{a,b})_x$ and $(F_{a,b})_y$ be Galois $U_3$-algebras as in (\ref{u3-algebra-1}) and (\ref{u3-algebra-2}), respectively. The Galois $U_3$-algebras $(F_{a,b})_x$ and $(F_{a,b})_y$ are isomorphic if and only if there exists $w\in F_{a,b}^\times$ such that
		\[w^p=xy,\qquad (\sigma_a-1)(\sigma_b-1)w=\zeta.\]
	\end{prop}
	
	\begin{proof}       
		(1) Since $Q_3=\ang{\tau}\simeq \Z/p\Z$ and $K^{Q_3}\simeq F_{a,b}$ as $(\Z/p\Z)^2$-algebras, we have an isomorphism of \'etale $F_{a,b}$-algebras $K\simeq (F_{a,b})_z$, for some $z\in F_{a,b}^\times$ such that $(\tau-1)z^{1/p}=\zeta^{-1}$. We may suppose that $K=(F_{a,b})_z$. 
		As $\tau$ commutes with $\sigma_b$ we have
		\[
		(\tau-1)(\sigma_b-1)z^{1/p}=(\sigma_b-1)(\tau-1)z^{1/p}=(\sigma_b-1)\zeta^{-1}=1,
		\]
		hence $(\sigma_b-1)z^{1/p}\in F_{a,b}^\times$. By Hilbert's Theorem 90 for the extension $F_{a,b}/F_a$, there is $t\in F_{a,b}^\times$
		such that $(\sigma_b-1)z^{1/p}=(\sigma_b-1)t$. Replacing $z$ by $zt^{-p}$,  we may thus assume  that $(\sigma_b-1)z^{1/p}=1$.
		In particular, $z\in F_a^\times$. Since $(\tau-1)z^{1/p}=\zeta^{-1}$, we have $\sigma_b\sigma_a(z^{1/p})=\zeta \sigma_a\sigma_b(z^{1/p})$. Thus
		\[(\sigma_b-1) (\sigma_a-1)z^{1/p} =(\sigma_b\sigma_a-\sigma_a\sigma_b+ (\sigma_a-1)(\sigma_b-1))z^{1/p}=\zeta (\sigma_a-1) (\sigma_b-1)z^{1/p} = \zeta,\]
		and hence   $(\sigma_a-1)z^{1/p} = b^{1/p}/\alpha'$ for some $\alpha'\in F_a^\times$. Moreover $N_a(\alpha'/\alpha)=b/b=1$, and so by Hilbert's Theorem 90 there exists $\theta\in F_a^{\times}$ such that $\alpha'/\alpha=(\sigma_a-1)\theta$. 
		We define $x\coloneqq z\theta^p\in F_a^\times$, and set $x^{1/p}\coloneqq z^{1/p}\theta\in (F_{a,b})_z^\times$. Then $K=(F_{a,b})_x$, where
		\[(\sigma_a-1)x^{1/p}=(\sigma_a-1)w\cdot (\sigma_a-1)\theta=\frac{b^{1/p}}{\alpha'}\cdot\frac{\alpha'}{\alpha}=\frac{b^{1/p}}{\alpha}\]
		and $(\sigma_b-1)x^{1/p}=1$, as desired.
		
		(2) The proof is analogous to that of (1). 
		
		(3) Suppose given an isomorphism of Galois $U_3$-algebras between $(F_{a,b})_x$ and $(F_{a,b})_y$. Let $t\in (F_{a,b})_x$ be the image of $y^{1/p}$ under the isomorphism and set
		\[w'\coloneqq x^{1/p} t \in (F_{a,b})_x.\] 
		Set $y'\coloneqq  t^p$.
		We have 
		$(\tau-1)w'=\zeta^{-1}\cdot\zeta=1$, and hence
		$w'\in F_{a,b}^\times$.  We have $(w')^p=xy'$. Since $F_b$ coincides with the $\ang{\sigma_a,\tau}$-invariant subalgebra of $(F_{a,b})_x$ and $(F_{a,b})_y$, the isomorphism $(F_{a,b})_y\to (F_{a,b})_x$ restricts to an isomorphism of Galois $\Z/p\Z$-algebras $F_b \to F_b$. Since the automorphism group of $F_b$ as a Galois $(\Z/p\Z)$-algebra is $\Z/p\Z$, generated by $\sigma_b$, this isomorphism $F_b \to F_b$ is equal to $\sigma_b^i$ for some $i\geq 0$. Thus $y'=\sigma_b^i(y)$. Define 
		\[w\coloneqq (w' a^{i/p})/\Prod_{j=0}^i\sigma_b^j(\beta)\in F_{a,b}^\times.\]
		We have
		\[(1-\sigma_b^i)y=(\Sum_{j=0}^i\sigma_b^j(1-\sigma_b))y=a^i/(\Prod_{j=0}^i\sigma_b^j(\beta^p))=w^p/(w')^p.\]
		Therefore
		\begin{equation}\label{w-eq1}w^p=(w')^p(1-\sigma_b^i)y=x\sigma_b^i(y)(1-\sigma_b^i)y=xy.\end{equation}
		We have $(\sigma_b-1)x^{1/p}=1$ and \[(\sigma_a-1)(\sigma_b-1)t=(\sigma_a-1)(\sigma_b-1)y^{1/p}=(\sigma_a-1)(a^{1/p}/\beta)=\zeta,\] therefore
		\[(\sigma_a-1)(\sigma_b-1)w'=(\sigma_a-1)(\sigma_b-1)t=\zeta.\]
		Since $(\sigma_a-1)(\sigma_b-1)a^{1/p}=1$ and $(\sigma_a-1)(\sigma_b-1)\beta=1$, we conclude that
		\begin{equation}\label{w-eq2}(\sigma_a-1)(\sigma_b-1)w=(\sigma_a-1)(\sigma_b-1)w'=1.\end{equation}
		Putting (\ref{w-eq1}) and (\ref{w-eq2}) together, we see that $w$ satisfies the conditions of (3).
		
		Conversely, suppose given $w'\in F_{a,b}^\times$ such that 
		\[xy=(w')^p,\qquad (\sigma_a-1)(\sigma_b-1)w'=\zeta.\] 
		
		\begin{claim}\label{claim-new-w}
			There exists $w\in F_{a,b}^\times$ such that \[xy=w^p,\qquad (\sigma_a-1)w=\zeta^{-i}\frac{b^{1/p}}{\alpha},\qquad (\sigma_b-1)w=\zeta^{-j}\frac{a^{1/p}}{\beta}\]
			for some integers $i$ and $j$.
		\end{claim}
		
		\begin{proof}[Proof of \Cref{claim-new-w}]
			We first find $\eta_a\in F_a^\times$ such that
			\begin{equation}\label{eta_a}
				\eta_a^p=1, \qquad (\sigma_a-1)(w'/\eta_a)=\zeta^{-i}\frac{b^{1/p}}{\alpha}.
			\end{equation}
			We have
			\[
			(\sigma_a - 1) (w')^p=(\sigma_a - 1) x=\frac{b}{\alpha^p}.   
			\]
			Let \[\zeta_a\coloneqq (\sigma_a-1)w'\cdot \alpha \cdot b^{-1/p}\in F_{a,b}^\times.\] We have $\zeta_a^p=1$. Moreover, $(\sigma_b-1)\zeta_a=\zeta\cdot 1\cdot\zeta^{-1}=1$,
			that is, $\zeta_a$ belongs to $F_a^\times$. If $F_a$ is a field, this implies that $\zeta_a=\zeta^i$ for some integer $i$, and (\ref{eta_a}) holds for $\eta_a=1$. 
			
			Suppose that $F_a$ is not a field. Then $F_a\simeq F^p$, where $\sigma_a$ acts by cyclically permuting the coordinates: \[\sigma_a(x_1,x_2,\dots,x_p)=(x_2,\dots,x_p, x_1).\] 
			We have $\zeta_a=(\zeta_1,\dots,\zeta_p)$ in $F_a=F^p$, where $\zeta_i\in F^\times$ is a $p$-th root of unity for all $i$. We have $N_a(\zeta_a)=N_a(\alpha)/b=1$, and so $\zeta_1\cdots\zeta_p=1$. Inductively define $\eta_1\coloneqq 1$ and $\eta_{i+1}\coloneqq \zeta_i\eta_i$ for all $i=1,\dots,p-1$. Then
			\[\eta_1/\eta_p=(\eta_1/\eta_2)\cdot(\eta_2/\eta_3)\cdots(\eta_{p-1}/\eta_p)=\zeta_1^{-1}\zeta_2^{-1}\cdots\zeta_{p-1}^{-1}=\zeta_p.\]
			Therefore the element $\eta_a\coloneqq (\eta_1,\dots,\eta_p)\in F^p=F_a$ satisfies $\eta_a^p=1$ and
			\[(\sigma_a-1)\eta_a=(\eta_2/\eta_1,\dots,\eta_p/\eta_{p-1},\eta_1/\eta_p)=(\zeta_1,\dots,\zeta_{p-1},\zeta_p)=\zeta_a.\]
			Thus
			\[\eta_a^p=1,\qquad (\sigma_a-1)(w'/\eta_a)=(\sigma_a-1)w'\cdot \zeta_a^{-1}=\frac{b^{1/p}}{\alpha}.\]
			Independently of whether $F_a$ is a field or not, we have found $\eta_a$ satisfying (\ref{eta_a}).
			
			Similarly, we construct $\eta_b\in F_b^\times$ such that
			\begin{equation}\label{eta_b}
				\eta_b^p=1,\qquad (\sigma_b-1)(w'/\eta_b)=\zeta^{-j}\frac{a^{1/p}}{\beta},
			\end{equation}
			for some integer $j$. Set $w\coloneqq w'/(\eta_a\eta_b)\in F_{a,b}^\times$. Putting together (\ref{eta_a}) and (\ref{eta_b}), we deduce that $w$ satisfies the conclusion of \Cref{claim-new-w}.
		\end{proof}
		
		Let $w\in F_{a,b}^\times$ be as in \Cref{claim-new-w}. By \Cref{minus-x}(1), applied $i$ times, the Galois $U_3$-algebra $(F_{a,b})_x$ is isomorphic to $(F_{a,b})_{a^ix}$, where
		\[(\sigma_a-1)(a^ix)^{1/p}=\frac{b^{1/p}}{\alpha},\qquad (\sigma_b-1)(a^ix)^{1/p}=1,\]
		By \Cref{minus-x}(2), applied $j$ times, the Galois $U_3$-algebra $(F_{a,b})_y$ is isomorphic to $(F_{a,b})_{b^jy}$, where
		\[(\sigma_a-1)(b^jy)^{1/p}=1,\qquad (\sigma_b-1)(b^jy)^{1/p}=\frac{a^{1/p}}{\beta}.\]
		It thus suffices to construct an isomorphism of $U_3$-algebras $(F_{a,b})_{a^ix}\simeq (F_{a,b})_{b^jy}$. Let  \[\tilde{w}\coloneqq wa^{i/p}b^{j/p}\in F_{a,b}^\times,\] so that
		\[(\sigma_a-1)\tilde{w}=\frac{a^{1/p}}{\beta},\qquad (\sigma_b-1)\tilde{w}=\frac{b^{1/p}}{\alpha}.\]
		Let $f\colon (F_{a,b})_{a^ix}\to (F_{a,b})_{b^jy}$ be the isomorphism of \'etale algebras which is the identity on $F_{a,b}$ and sends $(a^ix)^{1/p}$ to $\tilde{w}/(b^jy)^{1/p}$. Note that $f$ is well defined because \[(\tilde{w})^p=wa^ib^j=(a^ix)(b^jy).\]
		Moreover,
		\[(\sigma_a-1)(\tilde{w}/(b^jy)^{1/p})=\frac{a^{1/p}}{\beta}=(\sigma_a-1)(a^ix)^{1/p},\]
		\[(\sigma_b-1)(\tilde{w}/(b^jy)^{1/p})=\frac{b^{1/p}}{\alpha}\cdot \frac{\alpha}{b^{1/p}}=1=(\sigma_b-1)(a^ix)^{1/p},\]
		and hence $f$ is $U_3$-equivariant.
	\end{proof}
	
	\subsection{Galois \texorpdfstring{$\cl{U}_4$}{U4}-algebras}
	Let $a,b,c\in F^\times$ be such that $(a,b)=(b,c)=0$ in $\on{Br}(F)$. By \Cref{cup-norm}, we may fix $\alpha\in F_a^\times$ and $\gamma\in F_c^\times$ be such that $N_a(\alpha)=N_c(\gamma)=b$. We have $\on{Gal}(F_{a,b,c}/F)=\ang{\sigma_a,\sigma_b,\sigma_c}$. The projection map $\cl{U}_4\to (\Z/p\Z)^3$ is given by $\cl{e}_{12}\mapsto\sigma_a$, $\cl{e}_{23}\mapsto\sigma_b$, $\cl{e}_{34}\mapsto\sigma_c$. Its kernel $\cl{Q}_4\subset \cl{U}_4$ is isomorphic to $(\Z/p\Z)^2$, generated by $\cl{e}_{13}$ and $\cl{e}_{24}$.
	We define the following elements of $\cl{U}_4$:
	\[\sigma_a\coloneqq \cl{e}_{12},\qquad \sigma_b\coloneqq \cl{e}_{23},\qquad  \sigma_c\coloneqq \cl{e}_{34},\qquad \tau_{ab}\coloneqq \cl{e}_{13},\qquad \tau_{bc}\coloneqq \cl{e}_{24}.\]

	Let $x\in F_a^\times$ and $z\in F_c^\times$ be such that
	\begin{equation}\label{uu4-0}(\sigma_a-1)x=\frac{b}{\alpha^p},\qquad (\sigma_c-1)z=\frac{b}{\gamma^p},\end{equation}
	and consider the Galois $\cl{U}_4$-algebra $K\coloneqq (F_{a,b,c})_{x,z}$, where $\cl{U}_4$ acts on $F_{a,b,c}$ via the surjection onto $\on{Gal}(F_{a,b,c}/F)$, and 
	\begin{equation}\label{uu4-1}
		(\sigma_a-1)x^{1/p}=\frac{b^{1/p}}{\alpha},\qquad (\sigma_b-1)x^{1/p}=1,\qquad (\sigma_c-1)x^{1/p}=1,\end{equation}
	\begin{equation}\label{uu4-2}
		(\tau_{ab}-1)x^{1/p}=\zeta^{-1},\qquad (\tau_{bc}-1)x^{1/p}=1,
	\end{equation}
	\begin{equation}\label{uu4-3}
		(\sigma_a-1)(x')^{1/p}=1,\qquad (\sigma_b-1)(x')^{1/p}=1,\qquad   (\sigma_c-1)(x')^{1/p}=\frac{b^{1/p}}{\gamma},\end{equation}
	\begin{equation}\label{uu4-4}(\tau_{ab}-1)(x')^{1/p}=1,\qquad (\tau_{bc}-1)(x')^{1/p}=\zeta.
	\end{equation}
	Note that (\ref{uu4-2}) follows from (\ref{uu4-1}) and (\ref{uu4-4}) follows from (\ref{uu4-3}). We leave to the reader to check that the relations (\ref{relation0-bar})-(\ref{relation4-bar}) are satisfied.
	
	\begin{prop}\label{uu4-equiv}
		Let $a,b,c\in F^\times$ be such that $(a,b)=(b,c)=0$ in $\on{Br}(F)$. Fix $\alpha\in F_a^\times$ and $\gamma\in F_c^\times$ such that $N_a(\alpha)=N_c(\gamma)=b$. Let $K$ be a Galois $\cl{U}_4$-algebra such that $K^{\cl{Q}_4}\simeq F_{a,b,c}$ as $(\Z/p\Z)^3$-algebras. Then there exist $x\in F_a^\times$ and $x'\in F_c^\times$ such that $K\simeq (F_{a,b,c})_{x,x'}$ as Galois $\cl{U}_4$-algebras, where $\cl{U}_4$ acts on $(F_{a,b,c})_{x,x'}$ by (\ref{uu4-1})-(\ref{uu4-4}).
	\end{prop}
	
	\begin{proof}
		Let $H$ (resp. $H'$) be the subgroup of $\cl{U}_4$ generated by $\sigma_c$ and $\tau_{bc}$ (resp. $\sigma_b$ and $\tau_{ab}$), and let $S$ be the subgroup of $\cl{U}_4$ generated by $H$ and $H'$. Note that $K^H$ is a Galois $U_3$-algebra over $F$ such that $(K^H)^{Q_3}\simeq F_{a,b}$ as $(\Z/p\Z)^2$-algebras and $K^S\simeq F_b$ as $(\Z/p\Z)$-algebras. Thus by \Cref{uu3}(1) there exists $x\in F_a^\times$ such that $K^H\simeq (F_{a,b})_x$ as Galois $U_3$-algebras. Similarly, by \Cref{uu3}(2) there exists $x'\in F_c^\times$ such that $K^{H'}\simeq (F_{b,c})_{x'}$ as Galois $U_3$-algebras. Therefore $x$ satisfies (\ref{uu4-1}) and $x'$ satisfies (\ref{uu4-3}). We apply \Cref{pushout-g-algebras}(2) to (\ref{phi-phi'}). We obtain the isomorphisms of $\cl{U}_4$-algebras
		\[K\simeq K^H\otimes_{K^S}K^{H'}\simeq (F_{a,b,c})_{x,x'},\]
		where $(F_{a,b,c})_{x,x'}$ is the $\cl{U}_4$-algebra given by (\ref{uu4-1}) and (\ref{uu4-3}).
	\end{proof}
	
	\subsection{Galois \texorpdfstring{$U_4$}{U4}-algebras}
	
	Let $a,b,c\in F^\times$, and suppose that $(a,b)=(b,c)=0$ in $\on{Br}(F)$. We write $(\Z/p\Z)^3=\ang{\sigma_a,\sigma_b,\sigma_c}$ and view $F_{a,b,c}$ as a Galois $(\Z/p\Z)^3$-algebra over $F$ as in \Cref{kummer-sub}. 	The quotient map $U_4\to (\Z/p\Z)^3$ is given by $e_{12}\mapsto \sigma_a$, $e_{23}\mapsto \sigma_b$ and $e_{34}\mapsto \sigma_c$. The kernel $Q_4$ of this map is generated by $e_{13}$, $e_{24}$ and $e_{14}$ and is isomorphic to $(\Z/p\Z)^3$. 
	We define the following elements of $U_4$:
	\[
	\sigma_a\coloneqq e_{12},\qquad
	\sigma_b\coloneqq e_{23},\qquad
	\sigma_c\coloneqq e_{34},
	\]\[
	\tau_{ab}\coloneqq e_{13}=[\sigma_a,\sigma_b],\qquad \tau_{bc}\coloneqq e_{24}=[\sigma_b,\sigma_c],\qquad \rho\coloneqq e_{14}=[\sigma_a,\tau_{bc}]=[\tau_{ab},\sigma_c].
	\]
	
	\begin{prop}\label{uu4}
		Let $a,b,c\in F^\times$ be such that $(a,b)=(b,c)=0$ in $\on{Br}(F)$. Let $\alpha\in F_a^\times$ and $\gamma\in F_c^\times$ be such that $N_a(\alpha)=b$ and $N_c(\gamma)=b$. Let $K$ be a Galois $\cl{U}_4$-algebra such that $K^{\cl{Q}_4}\simeq F_{a,b,c}$
		as $(\Z/p\Z)^3$-algebras.
		
		There exists a Galois $U_4$-algebra $L$ over $F$ such that $L^{Z_4}\simeq K$ as $\cl{U}_4$-algebras if and only if there exist $u,u'\in F_{a,c}^\times$ such that
		\[\alpha\cdot (\sigma_a-1)u=\gamma\cdot(\sigma_c-1)u'.\] 
		and such that $K$ is isomorphic to the Galois $\cl{U}_4$-algebra $(F_{a,b,c})_{x,x'}$ determined by (\ref{uu4-1})-(\ref{uu4-4}), where $x=N_c(u)\in F_a^\times$ and  $x'=N_a(u')\in F_c^\times$.
	\end{prop}
	
	\begin{proof}
		Suppose that $K=(F_{a,b,c})_{x,x'}$, with $\cl{U}_4$-action determined by (\ref{uu4-1})-(\ref{uu4-4}). Let $L$ be a Galois $U_4$-algebra over $F$ be such that $L^{Z_4}=K$, and let  $y\in K^\times$ be such that $L=K_y$. 
		
		We have  $\on{Gal}(L/F_{a,b,c})=Q_4=\ang{\tau_{ab}, \tau_{bc},\rho}\simeq (\Z/p\Z)^3$, and hence one may choose $y$ in $F_{a,b,c}^\times$ and such that
		\[
		(\tau_{ab}-1)y^{1/p}=1,\qquad (\tau_{bc}-1)y^{1/p}=1,\qquad (\rho-1)y^{1/p}=\zeta^{-1}.
		\]
		The element $\sigma_b$ commutes with $\tau_{ab}, \tau_{bc}$ and $\rho$. Hence
		\[
		\tau_{ab}(\sigma_b-1)(y^{1/p})=(\sigma_b-1)\tau_{ab}(y^{1/p})=(\sigma_b-1)(y^{1/p}).
		\]
		Similarly
		\[
		\tau_{bc}(\sigma_b-1)(y^{1/p})=(\sigma_b-1)(y^{1/p})\]
		and \[\rho(\sigma_b-1)(y^{1/p})=(\sigma_b-1)(\zeta\cdot y^{1/p})=(\sigma_b-1)(y^{1/p}).\]
		It follows that $(\sigma_b-1)(y^{1/p})\in F_{a,b,c}^\times$. By Hilbert's Theorem 90, applied to
		$F_{a,b,c}/F_{a,c}$, there is $q\in F_{a,b,c}^\times$ such that
		$(\sigma_b-1)(y^{1/p})=(\sigma_b-1)q$. Replacing $y$ by $y/ q^p$, we may assume that $\sigma_b(y^{1/p})=y^{1/p}$.
		In particular, $y\in F_{a,c}^\times$. We have:
		\begin{align*}
			\rho(\sigma_a-1) (y^{1/p})= &\ (\sigma_a -1)\rho(y^{1/p})=(\sigma_a -1)(\zeta^{-1}\cdot y^{1/p})=(\sigma_a -1)(y^{1/p}), \\
			\sigma_b(\sigma_a-1) (y^{1/p})= &\ (\sigma_a\sigma_b{\tau_{ab}}^{-1} -\sigma_b)(y^{1/p})=(\sigma_a-1) (y^{1/p}),\\
			\tau_{ab}(\sigma_a-1) (y^{1/p})= &\ (\sigma_a -1)\tau_{ab}(y^{1/p})=(\sigma_a -1)(y^{1/p}),\\
			\tau_{bc}(\sigma_a-1) (y^{1/p})= &\ (\rho^{-1} \sigma_a -1)\tau_{bc}(y^{1/p})=( \sigma_a \rho^{-1} -1)(y^{1/p})=\zeta\cdot (\sigma_a-1) (y^{1/p}).
		\end{align*}
		By (\ref{uu4-3})-(\ref{uu4-4}), analogous identities are satisfied by $(x')^{1/p}$: 
		\[(\rho-1)(x')^{1/p}=(\sigma_b-1)(x')^{1/p}=(\tau_{ab}-1)(x')^{1/p}=1,\qquad (\tau_{bc}-1)(x')^{1/p}=\zeta.\]
		Therefore
		\[
		(\sigma_a-1) (y^{1/p})=\frac{(x')^{1/p}}{u'}
		\]
		for some $u'\in F_{a,c}^\times$. In particular, $x'= N_a(u')$. A similar computation shows that
		\[
		(\sigma_c-1) (y^{1/p})=\frac{x^{1/p}}{u}
		\]
		for some $u\in F_{a,c}^\times$. In particular, $x= N_c(u)$. In addition,
		\[
		\frac{b^{1/p}}{\alpha}=(\sigma_a-1) (x^{1/p})=(\sigma_a-1) [u\cdot (\sigma_c-1)(y^{1/p})],
		\]
		\[
		\frac{b^{1/p}}{\gamma}=(\sigma_c-1) ((x')^{1/p})=(\sigma_c-1) [u'\cdot (\sigma_a-1)(y^{1/p})].
		\]
		Therefore
		\[\alpha\cdot (\sigma_a-1)u=\gamma\cdot(\sigma_c-1)u'.\]
		
		Conversely, suppose given $u,u'\in F_{a,c}^\times$ such that
		\[\alpha\cdot(\sigma_a-1)u=\gamma\cdot(\sigma_c-1)u',\qquad x=N_c(u),\qquad x'=N_a(u').\]
		Then
		\[
		(\sigma_a-1)x=(\sigma_a-1)N_c(u)=N_c(\sigma_a-1)u=N_c\left(\frac{\gamma}{\alpha}\right)=\frac{b}{\alpha^p},
		\]
		\[
		(\sigma_c-1)x'=(\sigma_c-1)N_a(u')=N_a(\sigma_c-1)u'=N_a\left(\frac{\alpha}{\gamma}\right)=\frac{b}{\gamma^p}.
		\]
		We have
		\[
		N_c \left(\frac{x}{u^p}\right)=\frac{N_c(x)}{N_c(u^p)}=\frac{x^p}{x^p}=1,
		\]
		\[
		N_a \left(\frac{x'}{(u')^p}\right)=\frac{N_a(x')}{N_a((u')^p)}=\frac{(x')^p}{(x')^p}=1,
		\]
		\[
		(\sigma_a-1)\left(\frac{x}{u^p}\right)=\frac{b}{\alpha^p\cdot (\sigma_a-1)u^p}=
		\frac{b}{\gamma^p\cdot (\sigma_c-1)(u')^p}=(\sigma_c-1)\left(\frac{x'}{(u')^p}\right).
		\]
		By Hilbert's Theorem 90 applied to $F_{a,c}/F$, there is $y\in F_{a,c}^\times$ such that
		\[
		(\sigma_a-1)y=\frac{x'}{(u')^p}  \quad\text{and}\quad  (\sigma_c-1)y=\frac{x}{u^p}.
		\]
		We consider the \'etale $F$-algebra $L\coloneqq K_y$ and make it into a Galois $U_4$-algebra such that $L^{Z_4}=K$. It suffices to describe the $U_4$-action on $y^{1/p}$. We set:
		\[
		(\sigma_a-1)(y^{1/p})=\frac{(x')^{1/p}}{u'},\quad (\sigma_b-1)(y^{1/p})=1,\quad   (\sigma_c-1)(y^{1/p})=\frac{x^{1/p}}{u},
		\]
		One checks that this definition is compatible with the relations (\ref{relation0})-(\ref{relation3}), and hence that it makes $L$ into a Galois $U_4$-algebra such that $L^{Z_4}=K$.
	\end{proof}
	
	We use \Cref{uu4} to give an alternative proof for the Massey Vanishing Conjecture for $n=3$ and arbitrary $p$. 
	
	\begin{prop}\label{mvc-3}
		Let $p$ be a prime, let $F$ be a field, and let $\chi_1,\chi_2,\chi_3\in H^1(F,\Z/p\Z)$. The following are equivalent.
		\begin{enumerate}
			\item We have $\chi_1\cup\chi_2=\chi_2\cup\chi_3=0$ in $H^2(F,\Z/p\Z)$.
			\item The Massey product $\ang{\chi_1,\chi_2,\chi_3}\subset H^2(F,\Z/p\Z)$ is defined. 
			\item The Massey product $\ang{\chi_1,\chi_2,\chi_3}\subset H^2(F,\Z/p\Z)$ vanishes.
		\end{enumerate}
	\end{prop}
	
	\begin{proof}
		It is clear that (3) implies (2) and that (2) implies (1). We now prove that (1) implies (3). The first part of the proof is a reduction to the case when $\on{char}(F)\neq p$ and $F$ contains a primitive $p$-th root of unity, and follows \cite[Proposition 4.14]{minac2016triple}.
		
		Consider the short exact sequence
		\begin{equation}\label{q4-u4}1\to Q_4\to U_4\to (\Z/p\Z)^3\to 1,\end{equation}
		where the map $U_4\to (\Z/p\Z)^3$ comes from (\ref{central-exact-seq}). Recall from the paragraph preceding \Cref{uu4} that the group $Q_4$ is abelian. Therefore, the homomorphism $\chi\coloneqq(\chi_1,\chi_2,\chi_3)\colon \Gamma_F \to (\Z/p\Z)^3$ induces a pullback map \[H^2((\Z/p\Z)^3, Q_4)\to H^2(F, Q_4).\] We let $A\in H^2(F, Q_4)$ be the image of the class of (\ref{q4-u4}) under this map. By \Cref{dwyer}, for every finite extension $F'/F$ the Massey product $\ang{\chi_1,\chi_2,\chi_3}$ vanishes over $F'$ if and only if the restriction of $\chi$ to $\Gamma_{F'}$ lifts to $U_4$, and this happens if and only if $A$ restricts to zero in $H^2(F', Q_4)$. 
		
		When $\on{char}(F)=p$, we have $\on{cd}(F)\leq 1$ by \cite[\S 2.2, Proposition 3]{serre1997galois}. Therefore $H^2(F,Q_4)=0$ and hence $A=0$. Thus (1) implies (3) when $\on{char}(F)=p$. 
		
		Suppose that $\on{char}(F)\neq p$. There exists an extension $F'/F$ of prime-to-$p$ degree such that $F'$ contains a primitive $p$-th root of $1$. If (1) implies (3) for $F'$, then $A$ restricts to zero in $H^2(F', Q_4)$. By a restriction-corestriction argument, we deduce that $A$ vanishes, that is, (1) implies (3) for $F$. We may thus assume that $F$ contains a primitive $p$-th root of unity $\zeta$. 
		
		Let $a,b,c\in F^\times$ be such that $\chi_a=\chi_1$, $\chi_b=\chi_2$ and $\chi_c=\chi_3$ in $H^1(F,\Z/p\Z)$. Since $(a,b)=(b,c)=0$ in $\on{Br}(F)$, there exists $\alpha\in F_a^\times$ and $\gamma\in F_c^\times$ such that $N_a(\alpha)=N_c(\gamma)=b$. Since $N_{ac}(\gamma/\alpha)=N_c(\gamma)/N_a(\alpha)=1$ in $F_{ac}^\times$, by Hilbert's Theorem 90 there exists $t\in F_{a,c}^\times$ such that $\gamma/\alpha = (\sigma_a\sigma_c - 1)t$. Define $u,u'\in F_{a,c}^\times$ by $u\coloneqq \sigma_c(t)$ and $u'\coloneqq t^{-1}$. Then
		\[\alpha\cdot(\sigma_a-1)u=\alpha\cdot(\sigma_a\sigma_c-\sigma_c)t=
		\alpha\cdot(\sigma_a\sigma_c-1)t\cdot(\sigma_c-1)t^{-1}=\gamma\cdot(\sigma_c-1)u'.\]
		Let $x\coloneqq N_c(u)\in F_a^\times$ and $x'\coloneqq N_a(u')\in F_c^\times$. Since $\sigma_a\sigma_c=\sigma_c\sigma_a$ on $F_{a,c}^\times$, we have 
		\[(\sigma_a-1)x=N_c((\sigma_a-1)u)=N_c((\sigma_c-1)u'\cdot(\gamma/\alpha))=N_c(\gamma)/N_c(\alpha)=b/\alpha^p.\]
		Similarly, $(\sigma_c-1)x'=b/\gamma^p$. Therefore $x,x'$ satisfy (\ref{uu4-0}). Let $K\coloneqq (F_{a,b,c})_{x,x'}$ be the Galois $\cl{U}_4$-algebra over $F$, with the $\cl{U}_4$-action given by (\ref{uu4-1})-(\ref{uu4-4}). By \Cref{uu4-equiv}, there exists a Galois $U_4$-algebra $L$ over $F$ such that $L^{Z_4}\simeq (F_{a,b,c})_{x,x'}$ as $\cl{U}_4$-algebras. In particular, $L^{Q_4}\simeq F_{a,b,c}$ as $(\Z/p\Z)^3$-algebras. By \Cref{dwyer-cor}, we conclude that $\ang{a,b,c}$ vanishes.
	\end{proof}

	\subsection{Galois \texorpdfstring{$\cl{U}_5$}{U5}-algebras}\label{uu5-sec}
	Let $a,b,c,d\in F^\times$. We write $(\Z/p\Z)^4=\ang{\sigma_a,\sigma_b,\sigma_c,\sigma_d}$ and regard $F_{a,b,c,d}$ as a Galois $(\Z/p\Z)^4$-algebra over $F$ as in \Cref{kummer-sub}.
	
	\begin{prop}\label{uu5}
		Let $a,b,c,d \in F^\times$ be such that $(a,b)=(b,c)=(c,d)=0$ in $\on{Br}(F)$. Then the Massey product $\ang{a,b,c,d}$ is defined if and only if there exist $u\in F_{a,c}^\times$,
		$v\in F_{b,d}^\times$ and $w\in F_{b,c}^\times$ such that
		\[N_a(u)\cdot N_d(v) = w^p,\qquad (\sigma_b - 1)(\sigma_c - 1)w = \zeta.\]
	\end{prop}
	
	\begin{proof}
		Denote by $U_4^+$ and $U_4^-$ the top-left and bottom-right $4\times 4$ corners of $U_5$, respectively, and let $S\coloneqq U_4^+\cap U_4^-$ be the middle subgroup $U_3$. Let $Q_4^+$ and $Q_4^-$ be the kernel of the map $U_4^+\to (\Z/p\Z)^3$ and $U_4^-\to (\Z/p\Z)^3$, respectively, and let $P_4^+$ and $P_4^-$ be the kernel of the maps $U_4^+\to U_3$ and $U_4^-\to U_3$, respectively.

		Suppose $\langle a,b,c,d \rangle$ is defined. By \Cref{dwyer-cor}, there exists a $\overline{U}_5$-algebra $L$ such that $L^{\cl{Q}_5}\simeq F_{a,b,c,d}$ as $(\Z/p\Z)^4$-algebras. Using \Cref{cup-norm}, we fix $\alpha\in F_a^\times$ and $\gamma\in F_c^\times$ such that
		$N_a(\alpha)=b$ and $N_c(\gamma)=b$. By \Cref{uu4}, there exist $u,u'\in F_{a,c}^\times$ such that, letting $x'\coloneqq N_c(u')$ and $x\coloneqq N_a(u)$, the $\cl{U}_4^+$-algebra $K_1$ induced by $L$ is isomorphic to the $\cl{U}_4^+$-algebra $(F_{a,b,c})_{x',x}$, where $\cl{U}_4^+$ acts via (\ref{uu4-1})-(\ref{uu4-4}), and where the roles of $x$ and $x'$ have been switched.
		
		Similarly, there exist $v,v'\in F_{b,d}^\times$ such that, letting $z\coloneqq N_d(v)$ and $z'\coloneqq N_b(v')$, the $\cl{U}^-_4$-algebra $K_2$ induced by $L$ is isomorphic to $(F_{b,c,d})_{z,z'}$. Since the $U_3$-algebras $(K_1)^{P_4^+}$ and $(K_2)^{P_4^-}$ are equal, by \Cref{uu3}(3) there exists $w\in F_{b,c}^\times$ such that 
		\[
		N_a(u)\cdot N_d(v) = xz = w^p,\qquad (\sigma_b -1)(\sigma_c -1)w = \zeta.
		\]
		Conversely, let $u\in F_{a,c}^\times$, $v\in F_{b,d}^\times$, and $w\in F_{b,c}^\times$ be such that
		\[
		N_a(u)\cdot N_d(v) = w^p,\qquad (\sigma_b -1)(\sigma_c -1)w = \zeta.\]
		By \Cref{cup-norm}, there exist $\alpha\in F_a^\times$ and $\delta\in F_d^\times$ such that $N_a(\alpha)=b$ and $N_d(\delta)=c$. We may write
		\[
		(\sigma_b - 1)w = \frac{c^{1/p}}{\beta},\qquad (\sigma_c - 1)w = \frac{b^{1/p}}{\gamma}.
		\]
		For some $\beta\in F_b^\times$ and $\gamma\in F_c^\times$. We have
		\[
		N_a((\sigma_c - 1)u\cdot (\gamma/\alpha)) = (\sigma_c - 1)N_a(u)\cdot N_a(\gamma/\alpha)= (\sigma_c - 1)w^p \cdot (\gamma^p/b) = 1.
		\]
		By Hilbert's Theorem 90, there is $u'\in F_{a,c}^\times$ such that
		\[
		\alpha\cdot (\sigma_a - 1)u' = \gamma\cdot (\sigma_c - 1)u.
		\]
		By \Cref{uu4}, we obtain a Galois $U_4^+$-algebra $K_1$ over $F$ with the property that $(K_1)^{Q_4^+}\simeq F_{a,b,c}$ as $(\Z/p\Z)^3$-algebras. Similarly, we get a Galois $U_4^-$-algebra over $F$ such that $(K_2)^{Q_4^-}\simeq F_{b,c,d}$ as $(\Z/p\Z)^3$-algebras. Since $N_a(u)\cdot N_d(v) = w^p$,  by \Cref{uu3}(3) the $U_3$-algebras $(K_1)^{P_4^+}$ and $(K_2)^{P_4^-}$ are isomorphic. Now \Cref{pushout-g-algebras} applied to the cartesian square (\ref{phi-phi'}) for $n=4$ yields a $\cl{U}_5$-Galois algebra $L$ such that $L^{Q_5}\simeq F_{a,b,c,d}$ as $(\Z/p\Z)^4$-algebras. By \Cref{dwyer-cor}, this implies that $\ang{a,b,c,d}$ is defined.
	\end{proof}

	\begin{lemma}\label{b-1-c-1}
		Let $b,c\in F^\times$ and $w\in F_{b,c}^{\times}$. Then $(\sigma_b-1)(\sigma_c-1)w=1$ if and only if there exist $w_b\in F_b^\times$ and $w_c\in F_c^\times$ such that $w=w_bw_c$ in $F_{b,c}^\times$.
	\end{lemma}
	
	\begin{proof}
		We have $(\sigma_b-1)(\sigma_c-1)(w_bw_c)=(\sigma_b-1)w_c=1$ for all $w_b\in F_b^\times$ and $w_c\in F_c^\times$. Conversely, if $w\in F_{b,c}^\times$ satisfies $(\sigma_b-1)(\sigma_c-1)w=1$, then $(\sigma_c-1)w\in F_c^\times$ and $N_c((\sigma_c-1)w)=1$, and hence by Hilbert's Theorem 90 there exists $w_c\in F_c^\times$ such that $(\sigma_c-1)w_c=(\sigma_c-1)w$. Letting $w_b\coloneqq w/w_c\in F_{b,c}^\times$, we have \[(\sigma_c-1)w_b=(\sigma_c-1)(w/w_c)=1,\] that is, $w_b\in F_b^{\times}$.
	\end{proof}

	From \Cref{uu5}, we derive the following necessary condition for a fourfold Massey product to be defined.
	
	\begin{prop}\label{condition-fixed-w}
		Let $p$ be a prime, let $F$ be a field of characteristic different from $p$ and containing a primitive $p$-th root of unity $\zeta$, let $a,b,c,d\in F^\times$, and suppose that $\ang{a,b,c,d}$ is defined over $F$. For every $w\in F_{b,c}^\times$ such that $(\sigma_b-1)(\sigma_c-1)w=\zeta$, there exist $u\in F_{a,c}^\times$ and $v\in F_{b,d}^\times$ such that $N_a(u)N_d(v)=w^p$.
	\end{prop}
	
	\begin{proof}
		By \Cref{uu5}, there exist $u_0\in F_{a,c}^\times$, $v_0\in F_{b,d}^\times$ and $w_0\in F_{a,c}^\times$ such that
		\[N_a(u_0)N_d(v_0) =w_0^p,\qquad (\sigma_b-1)(\sigma_c-1)w_0=\zeta.\]
		We have $(\sigma_b-1)(\sigma_c-1)(w_0/w)=1$. By \Cref{b-1-c-1}, this implies that $w_0=w w_b w_c$, where $w_b\in F_b^\times$ and $w_c\in F_c^\times$. If we define $u=u_0 w_c$ and $v=v_0 w_b$, then \[N_a(u)N_d(v)=N_a(u_0)N_a(w_c)N_d(v_0)N_d(w_b)=w_0^pw_c^pw_b^p=w^p.\qedhere\]
	\end{proof}
	
	\section{A generic variety}\label{generic-variety-section}
	In this section, we let $p$ be a prime number, and we let $F$ be a field of characteristic different from $p$ and containing a primitive $p$-th root of unity $\zeta$. 
	
	Let $b,c\in F^\times$, and let $X$ be the Severi-Brauer variety associated to $(b,c)$ over $F$; see \cite[Chapter 5]{gille2017central}. For every \'etale $F$-algebra $K$, we have $(b,c)=0$ in $\on{Br}(K)$ if and only if $X_K\simeq \P^{p-1}_K$ over $K$. In particular, $X_b\simeq \P^{p-1}_b$ over $F_b$. By \cite[Theorem 5.4.1]{gille2017central}, the central simple algebra $(b,c)$ is split over $F(X)$.
	
	We define the degree map $\deg\colon \on{Pic}(X)\to \Z$ as the composition of the pullback map $\on{Pic}(X)\to \on{Pic}(X_b)\simeq\on{Pic}(\P^{p-1}_b)$ and the degree isomorphism $\on{Pic}(\P^{p-1}_b)\to\Z$. This does not depend on the choice of isomorphism $X_b\simeq \P^{p-1}_b$.
	
	\begin{lemma}\label{construct-w}
		Let $b,c\in F^\times$, let $G_b\coloneqq \on{Gal}(F_b/F)$, and let $X$ be the Severi-Brauer variety of $(b,c)$ over $F$. Let $s_1,\dots,s_p$ be homogeneous coordinates on $\P^{p-1}_F$.
		
		(1) There exists a $G_b$-equivariant isomorphism $X_b\xrightarrow{\sim} \P_b^{p-1}$, where $G_b$ acts on $X_b$ via its action on $F_b$, and on $\P^{p-1}_b$ by
		\[\sigma_b^*(s_1)=cs_p,\qquad \sigma_b^*(s_i)=s_{i-1}\quad (i=2,\dots,p).\]
		
		(2) If $(b,c)\neq 0$ in $\on{Br}(F)$, the image of $\deg\colon \on{Pic}(X)\to\Z$ is equal to $p\Z$.
		
		(3) There exists a rational function
		$w\in F_{b,c}(X)^\times$ such that 
		\[(\sigma_b-1)(\sigma_c-1)w=\zeta\] and \[\on{div}(w)=x-y\qquad \text{in $\on{Div}(X_{b,c})$},\] where $x,y\in (X_{b,c})^{(1)}$ satisfy $\deg(x)=\deg(y)=1$, $\sigma_b(x)=x$ and $\sigma_c(y)=y$.
	\end{lemma}	
	
	\begin{proof}
		(1) Consider the $1$-cocycle $z\colon G_b\to \on{PGL}_p(F_b)$ given by
		\[\sigma_b\mapsto
		\begin{bmatrix}
			0 & 0 & \dots & 0 & c \\
			1 & 0 & \dots & 0 & 0 \\
			0 & 1 & \dots & 0 & 0 \\
			\vdots & \vdots & \ddots & \vdots & \vdots \\
			0 & 0 & \dots & 1 & 0 \\
		\end{bmatrix}.
		\]
		By \cite[Construction 2.5.1]{gille2017central}, the class $[z]\in H^1(G_b,\on{PGL}_p(F_b))$ coincides with the class of the degree-$p$ central simple algebra over $F$ with Brauer class $(b,c)$, and hence with the class of the associated Severi-Brauer variety $X$. It follows that we have a $G_b$-equivariant isomorphism $X_b\simeq \P^{p-1}_b$, where $G_b$ acts on $X_b$ via its action on $F_b$, and on $\P^{p-1}_b$ via the cocycle $z$. This proves (1). 
		
		(2) By a theorem of Lichtenbaum \cite[Theorem 5.4.10]{gille2017central}, we have an exact sequence
		\[\on{Pic}(X)\xrightarrow{\deg} \Z\xrightarrow{\delta}\Br(F),\]
		where $\delta(1)=(b,c)$. Since $(b,c)$ has exponent $p$, we conclude that the image of $\deg$ is equal to $p\Z$.
		
		(3) Let $G_{b,c}\coloneqq \on{Gal}(F_{b,c}/F)=\ang{\sigma_b,\sigma_c}$. By (1), there is a $G_{b,c}$-equivariant isomorphism $f\colon \P^{p-1}_{b,c}\to X_{b,c}$, where $G_{b,c}$ acts on $X_{b,c}$ via its action on $F_{b,c}$, the action of $\sigma_c$ on $\P^{p-1}_{b,c}$ is trivial and the action of $\sigma_b$ on $\P^{p-1}_{b,c}$ is determined by
		\[\sigma_b^*(s_1)=cs_p,\qquad \sigma_b^*(s_i)=s_{i-1}\quad (i=2,\dots,p).\]
		Consider the linear form $l\coloneqq\Sum_{i=1}^{p} c^{i/p}\cdot s_i$ on $\P^{p-1}_{b,c}$ and set $w'\coloneqq l/s_p\in F_{b,c}(\P^{p-1})^\times$.
		We have $\sigma_b^*(l)=c^{1/p}\cdot l$, and hence $(\sigma_b-1)w'=c^{1/p}\cdot (s_p/s_{p-1})$. It follows that $(\sigma_c-1)(\sigma_b-1)w'=\xi$.
		Let $x',y'\in \on{Div}(\P^{p-1}_{b,c})$ be the classes of linear subspaces of $\P^{p-1}_{b,c}$ given by $l=0$ and $s_p=0$, respectively. Then
		\[\on{div}(w')=x'-y',\qquad \sigma_b(x')=x',\qquad \sigma_c(y')=y'.\]
		Define 
		\[w\coloneqq w'\circ f^{-1}\in F_{b,c}(X)^\times,\qquad x'\coloneqq f_*(x)\in (X_{b,c})^{(1)},\qquad y'\coloneqq f_*(y)\in (X_{b,c})^{(1)}.\]
		Then $w,x,y$ satisfy the conclusion of (3).
	\end{proof} 
	
	\begin{lemma}\label{weil-restrictions-exact}
		Let $a,b,c,d\in F^\times$. The complex of tori
		\[R_{a,c}(\G_{\on{m}})\times R_{b,d}(\G_{\on{m}})\xrightarrow{\varphi}R_{b,c}(\G_{\on{m}})\xrightarrow{\psi} R_{b,c}(\G_{\on{m}}),\]
		where $\varphi(u,v)\coloneqq N_a(u)N_d(v)$ and $\psi(z)=(\sigma_b-1)(\sigma_c-1)z$, is exact.
	\end{lemma}
	
	\begin{proof}
		By \Cref{b-1-c-1}, we have an exact sequence
		\[R_c(\G_{\on{m}})\times R_b(\G_{\on{m}})\xrightarrow{\varphi'} R_{b,c}(\G_{\on{m}})\xrightarrow{\psi} R_{b,c}(\G_{\on{m}}),\]
		where $\varphi'(x,y)=xy$. The homomorphism $\varphi$ factors as
		\[R_{a,c}(\G_{\on{m}})\times R_{b,d}(\G_{\on{m}})\xrightarrow{N_a\times N_d}R_c(\G_{\on{m}})\times R_b(\G_{\on{m}})\xrightarrow{\varphi'} R_{b,c}(\G_{\on{m}}).\]
		Since the homomorphisms $N_a$ and $N_d$ are surjective, so is $N_a\times N_d$. We conclude that $\on{Im}(\varphi)=\on{Im}(\varphi')=\on{Ker}(\psi)$, as desired.
	\end{proof}
	
	Let $a,b,c,d\in F^\times$, and consider the complex of tori of \Cref{weil-restrictions-exact}. We define the following groups of multiplicative type over $F$:
	\[P\coloneqq R_{a,c}(\G_{\on{m}})\times R_{b,d}(\G_{\on{m}}),\qquad S\coloneqq \on{Ker}(\psi)=\on{Im}(\varphi),\qquad T\coloneqq \on{Ker}(\varphi)\subset P.\]
	By \Cref{weil-restrictions-exact}, we get a short exact sequence
	\begin{equation}\label{t-p-s-2}
		1\to T\xrightarrow{\iota} P\xrightarrow{\pi} S\to 1,
	\end{equation}
	where $\iota$ is the inclusion map and $\pi$ is induced by $\varphi$.

	\begin{lemma}\label{t-is-torus}
		The groups of multiplicative type $T$, $P$ and $S$ are tori.
	\end{lemma}
	
	\begin{proof}
		It is clear that $P$ and $S$ are tori. We now prove that $T$ is a torus. Consider the subgroup $Q\subset R_{a,c}(\G_{\on{m}})$ which makes the following commutative square cartesian:
		\begin{equation}\label{cartesian-q}
			\begin{tikzcd}
				Q \arrow[r,hook] \arrow[d] & R_{a,c}(\G_{\on{m}}) \arrow[d,"N_a"] \\
				\G_{\on{m}} \arrow[r,hook] & R_c(\G_{\on{m}}).
			\end{tikzcd}
		\end{equation}
		Here the bottom horizontal map is the obvious inclusion. It follows that $Q$ is an $R_c(R^{(1)}_a(\G_{\on{m}}))$-torsor over $\G_{\on{m}}$, and hence it is smooth and connected. Therefore $Q$ is a torus.
		
		The image of the projection $T \xhookrightarrow{\iota} P\to R_{a,c}(\G_{\on{m}})$ is contained in the torus $Q$. Moreover, the kernel $U$ of the projection is $R_b(R^{(1)}_{F_{b,d}/F_b}(\G_{\on{m}}))$, and hence it is also a torus. We have an exact sequence
		\[1\to U\to T\to Q.\]
		We have $\dim(U) = p(p-1)$. From (\ref{t-p-s-2}), we see that $\dim (T) = 2p^2 -2p+1$, and from (\ref{cartesian-q}) that $\dim (Q) = p^2 - p +1$. Therefore $\dim (T) = \dim (U) + \dim (Q)$, and so the sequence 
		\[1 \to U \to T \to Q \to 1\]
		is exact. As $U$ and $Q$ are tori, so is $T$.
	\end{proof}
	
	\begin{prop}\label{generic-var-generic}
		Let $p$ be a prime, let $F$ be a field of characteristic different from $p$ and containing a primitive $p$-th root of unity $\zeta$, and let $a,b,c,d\in F^\times$. Suppose that $(a,b)=(b,c)=(c,d)=0$ in $\on{Br}(F)$, and let $w\in F_{b,c}^\times$ be such that
		$(\sigma_b-1)(\sigma_c-1)w = \zeta$. Let $T$ and $P$ be the tori appearing in (\ref{t-p-s-2}), and let $E_w\subset P$ be the $T$-torsor given by the equation $N_a(u)N_d(v)=w^p$.
		Then the mod $p$ Massey product $\ang{a,b,c,d}$ is defined if and only if $E_w$ is trivial.
	\end{prop}
	
	The construction of $E_w$ is functorial in $F$. Therefore, for every field extension $K/F$, the mod $p$ Massey product $\ang{a,b,c,d}$ is defined if and only if $E_w$ is split by $K$. We may thus call $E_w$ a generic variety for the property ``the Massey product $\ang{a,b,c,d}$ is defined.''
	
	\begin{proof}
		By \Cref{condition-fixed-w}, the Massey product $\ang{a,b,c,d}$ is defined over $F$ if and only if there exist $u\in F_{a,c}^\times$ and $v\in F_{b,d}^\times$ such that the equation $N_a(u)N_d(v)=w^p$ has a solution over $F$, that is, if and only if the $T$-torsor $E_w$ is trivial.
	\end{proof}
	
	\begin{cor}\label{generic-var}
		Let $p$ be a prime, let $F$ be a field of characteristic different from $p$ and containing a primitive $p$-th root of unity $\zeta$, and let $a,b,c,d\in F^\times$. Let $X$ be the Severi-Brauer variety of $(b,c)$ over $F$, fix $w\in F_{b,c}(X)^\times$ as in \Cref{construct-w}(3), and let $E_w\subset P_{F(X)}$ be the $T_{F(X)}$-torsor given by the equation $N_a(u)N_d(v)=w^p$.
		
		Then $\ang{a,b,c,d}$ is defined over $F(X)$ if and only if $E_w$ is trivial over $F(X)$.
	\end{cor}
	
	\begin{proof}
		This is a special case of \Cref{generic-var-generic}, applied over the ground field $F(X)$.
	\end{proof}
	
	\section{Proof of Theorem \ref{main-explicit}}\label{section-5}

	Let $p$ be a prime, and let $F$ be a field of characteristic different from $p$ and containing a primitive $p$-th root of unity $\zeta$. Let $a,b,c,d\in F^\times$ be such that their cosets in $F^\times/F^{\times p}$ are $\F_p$-linearly independent. 
	Consider the field $K\coloneqq F_{a,b,c,d}$, and write $G=\on{Gal}(K/F)=\ang{\sigma_a,\sigma_b,\sigma_c,\sigma_d}$ as in \Cref{kummer-sub}. We set $N_a\coloneqq \Sum_{j=0}^{p-1}\sigma_a^j\in \Z[G]$. For every subgroup $H$ of $G$, we also write $N_a$ for the image of $N_a\in \Z[G]$ under the canonical map $\Z[G]\to \Z[G/H]$. We define $N_b$, $N_c$ and $N_d$ in a similar way.
	
	Let
	\[1\to T\xrightarrow{\iota} P\xrightarrow{\pi} S\to 1\]
	be the short exact sequence of $F$-tori (\ref{t-p-s-2}). It induces a short exact sequence of cocharacter $G$-lattices 
	\[0\to T_*\xrightarrow{\iota_*} P_*\xrightarrow{\pi_*} S_*\to 1.\]
	By definition of $P$ and $S$, we have 
	\[P_*=\Z[G/\langle \sigma_b,\sigma_d\rangle]\oplus \Z[G/\langle \sigma_a,\sigma_c\rangle],\qquad S_*=\langle N_b,N_c\rangle\subset \Z[G/\langle\sigma_a,\sigma_d\rangle ].\]
	Let $X$ be the Severi-Brauer variety associated to $(b,c)\in \on{Br}(F)$. Since $X_K\simeq \P^{p-1}_K$, the degree map $\on{Pic}(X_K)\to \Z$ is an isomorphism, and so the map $\on{Div}(X_K)\to \on{Pic}(X_K)$ is identified with the degree map $\deg\colon \on{Div}(X_K)\to \Z$. The sequence (\ref{4exact-t}) for the torus $T$ thus takes the form
	\begin{equation}\label{4exact-t-deg}1\to T(K) \to T(K(X))\xrightarrow{\on{div}}\on{Div}(X_K)\otimes T_*\xrightarrow{\deg} T_*\to 0,
	\end{equation}
	where $T_*$ denotes the cocharacter lattice of $T$.
	
	\begin{lemma}\label{invar}
		(1)  We have $(T_*)^G=\Z\cdot\eta$, where $\iota_*(\eta)=(N_a N_c,-N_b N_d)$ in $(P_*)^G$.
		
		(2) If $(b,c)\neq 0$ in $\on{Br}(F)$, the image of $\deg\colon (\on{Div}(X_{b,c})\otimes T_*)^G\to (T_*)^G$ is equal to $p(T_*)^G$.
	\end{lemma}

	\begin{proof}
		(1)	The free $\Z$-module $(P_*)^G$ has a basis consisting of the elements $(N_aN_c,0)$ and $(0,N_bN_d)$. The map $\pi_*\colon P_* \to S_*\subset \Z[G/\ang{\sigma_a,\sigma_d}]$
		takes $(1, 0)$ to $N_b$ and $(0, 1)$ to $N_c$. It follows that $\on{Ker}(\pi_*)^G$ is generated by $(N_aN_c, -N_bN_d)$.
		
		(2) By Lemma \ref{construct-w}(2), the image of
		the composition
		\[\on{Div}(X)\otimes T_*^G=(\on{Div}(X)\otimes T_*)^G\to(\on{Div}(X_{b,c})\otimes T_*)^G\xrightarrow{\on{\deg}} (T_*)^G\]
		is equal to $p(T_*)^G$.
		Thus the image of the degree map contains $p(T_*)^G$. We now show that the image the degree map is contained in $p(T_*)^G$.
		
		For every $x\in X^{(1)}$, pick $x'\in (X_{b,c})^{(1)}$ lying over $x$, and write $H_x$ for the $G$-stabilizer of $x'$. The injective homomorphisms of $G$-modules
		\[j_x\colon \Z[G/H_x]\hookrightarrow \on{Div}(X_{b,c}),\qquad gH_x\mapsto g(x')\]
		yield an isomorphism of $G$-modules 
		\[\oplus_{x\in X^{(1)}}j_x\colon\oplus_{x\in X^{(1)}}\Z[G/H_x]\xrightarrow{\sim} \on{Div}(X_{b,c}).\] 
		In order to conclude, it suffices to show that the image of
		\begin{equation}\label{degree}
			(T_*)^{H_x}=(\Z[G/H_x]\otimes T_*)^G\to (\on{Div}(X_{b,c})\otimes T_*)^G\xrightarrow{\deg} (T_*)^G
		\end{equation}
		is contained in $p(T_*)^G$ for all $x\in X^{(1)}$. Set $H:=H_x$.
		
		The composition (\ref{degree}) takes a cocharacter $q\in (T_*)^{H}$ to
		\[\deg(\Sum_{gH\in G/H} gx'\otimes gq)=\deg(x')\cdot N_{G/H}(q).\]
		Thus (\ref{degree})
		coincides with the norm map $N_{G/H}$ times the degree of $x'$.
		
		Suppose that $G=H$. Then $\deg(x')=\deg(x)$ and, since $(b,c)\neq 0$, the degree of $x$ is divisible by $p$ by Lemma \ref{construct-w}(2).
		
		Suppose that $G\neq H$. Then either $\langle \sigma_a,\sigma_c\rangle$ or $\langle \sigma_b,\sigma_d\rangle$
		is not contained in $H$. Suppose $\langle \sigma_b,\sigma_d\rangle$ is not in $H$, and let $N$ be the subgroup generated by $H,\sigma_b,\sigma_d$. Note that $H$ is a proper subgroup of $N$.
		
		The norm map $N_{G/H}:(T_*)^H\to (T_*)^G$ is  the composition of the two norm maps
		\[
		(T_*)^H\xrightarrow{N_{N/H}} (T_*)^N\xrightarrow{N_{G/N}} (T_*)^G.
		\]
		Since $\Z[G/\ang{\sigma_b,\sigma_d}]^H=\Z[G/\ang{\sigma_b,\sigma_d}]^N$, the norm map $(T_*)^H\to (T_*)^N$ is multiplication by
		$[N:H]\in p\Z$ on the first component of $T_*$ with respect to the inclusion $\iota_*$ of $T_*$ into $P_*=\Z[G/\ang{\sigma_b,\sigma_d}]\oplus\Z[G/\ang{\sigma_a,\sigma_c}]$. 
		
		By Lemma \ref{invar}(1), $(T_*)^G=\Z\cdot\eta$, where $\iota_*(\eta)=(N_a N_c,-N_b N_d)$ in $(P_*)^G$.
		Since $N_a N_c$ is not divisible by $p$ in $\Z[G/\ang{\sigma_b,\sigma_d}]$, the image of (\ref{degree})
		is contained in $p\Z\cdot\eta=p(T_*)^G$, as desired.
	\end{proof}
	
	We write 
	\[\cl{\eta} \in \on{Coker}[(\on{Div}(X_{b,c})\otimes T_*)^G\xrightarrow{\deg} (T_*)^G]\]
	for the coset of the generator $\eta\in (T_*)^G$ appearing in \Cref{invar}(1). If $(b,c)\neq 0$, then we have $\cl{\eta}\neq 0$ by \Cref{invar}(2). We consider the subgroup of unramified torsors     \[H^1(G,T(K(X)))_{\on{nr}}\coloneqq \on{Ker}[H^1(G,T(K(X)))\xrightarrow{\on{div}}H^1(G,\on{Div}(X_K\otimes T_*))],\]
	and the homomorphism 
	\[\theta\colon H^1(G,T(K(X)))_{\on{nr}}\to \on{Coker}[\on{Div}(X_K)\otimes T_*\xrightarrow{\deg} T_*],\]
	which are defined in (\ref{theta-t}).
	
	\begin{lemma}\label{keyprop}
		Let $b,c\in F^\times$ be such that $(b,c)\neq 0$ in $\on{Br}(F)$, let $w\in F_{b,c}(X)^{\times}$ be such that $(\sigma_b-1)(\sigma_c-1)w=\zeta$ and $\on{div}(w)=x-y$, where $\deg(x)=\deg(y)=1$ and $\sigma_b(x)=x$ and $\sigma_c(y)=y$. Let $E_w\subset P_{F(X)}$ be the $T_{F(X)}$-torsor given by the equation $N_a(u)N_d(v)=w^p$, and write $[E_w]$ for the class of $E_w$ in $H^1(G,T(K(X)))$.
		
		(1) We have $[E_w]\in H^1(G,T(K(X)))_{\on{nr}}$.
		
		(2) Let $\theta$ be the homomorphism of (\ref{theta-t}). We have $\theta([E_w])=-\cl{\eta}\neq 0$.
	\end{lemma}
	
	\begin{proof}
		The $F$-tori $T$, $P$ and $S$ of (\ref{t-p-s-2}) are split by $K=F_{a,b,c,d}$. Therefore, we may consider diagram (\ref{big-diagram}) for the short exact sequence (\ref{t-p-s-2}), the splitting field $K/F$, and $X$ the Severi-Brauer variety of $(b,c)$ over $F$:
		\[
		\begin{tikzcd}    & \on{Div}((X_K)\otimes T_*)^G \arrow[d,hook,"\iota_*"] \arrow[r,"\deg"] & (T_*)^G \arrow[d,hook,"\iota_*"] \\
			P(F(X)) \arrow[d,"\pi_*"] \arrow[r,"{\on{div}}"] & (\on{Div}(X_K)\otimes P_*)^G \arrow[d,"\pi_*"] \arrow[r,"\deg"] & (P_*)^G \arrow[d,"\pi_*"] \\
			S(F(X)) \arrow[d,twoheadrightarrow,"\partial"] \arrow[r,"{\on{div}}"] & (\on{Div}(X_K)\otimes S_*)^G \arrow[d,"\partial"] \arrow[r,"\deg"] & (S_*)^G \\
			H^1(G,T(K(X))) \arrow[r,"{\on{div}}"] & H^1(G,\on{Div}(X_K\otimes T_*)).
		\end{tikzcd}
		\] 
		Since $(\sigma_b-1)(\sigma_c-1)w^p=1$, we have $w^p\in S(F(X))$.
		The image of $w^p$ under $\partial$ is equal to $[E_w]\in H^1(G,T(K(X)))$.
		
		Let $H\subset G$ be the subgroup generated by $\sigma_a$ and $\sigma_d$. The canonical isomorphism
		\[
		\on{Div}(X_{b,c})=\on{Div}(X_K)^H=(\on{Div}(X_K)\otimes \Z[G/H])^G
		\]
		sends the divisor $\on{div}(w)=x-y$ to $\Sum_{i,j}\sigma_b^i\sigma_c^j (x-y)\otimes \sigma_b^i\sigma_c^j$. Therefore, the element $\on{div}(w^p)$ in $(\on{Div}(X_K)\otimes S_*)^G\subset (\on{Div}(X_K)\otimes \Z[G/H])^G$
		is equal to
		\[
		e:=p\Sum_{i,j=0}^{p-1}(\sigma_b^i\sigma_c^j(x-y)\otimes \sigma_b^i\sigma_c^j)=
		p\Sum_{j=0}^{p-1}(\sigma_c^j x \otimes \sigma_c^j N_b)-p\Sum_{i=0}^{p-1}(\sigma_b^i y \otimes \sigma_b^i N_c).
		\]
		Since $S_*$ is the sublattice of $\Z[G/\langle \sigma_a,\sigma_d\rangle]$ generated by $N_b$ and $N_c$, this implies that $e$ belongs to $S_*$. Then $e=\pi_*(f)$, where
		\[
		f\coloneqq\Sum_{j=0}^{p-1}(\sigma_c^j x \otimes \sigma_c^j N_a)-\Sum_{i=0}^{p-1}(\sigma_b^i y \otimes \sigma_b^i N_d)\in (\on{Div}(X_K)\otimes P_*)^G.
		\]
		It follows that $\on{div}(E_w)=\partial(e)=\partial(\pi_*(f))=0$, which proves (1).
		
		Moreover, since $\deg(x)=\deg(y)=1$ we have
		\[
		\deg(f)=(N_a N_c, - N_b N_d)=\iota_*(\eta)\qquad\text{in $(P_*)^G$.}
		\]
		In view of (\ref{theta'}), this implies that $\theta([E_w])=-\cl{\eta}$. We know from \Cref{invar}(2) that $\cl{\eta}\neq 0$. This completes the proof of (2).
	\end{proof}
	
	\begin{proof}[Proof of Theorem \ref{main-explicit}]
		Replacing $F$ by a finite extension if necessary, we may suppose that $F$ contains a primitive $p$-th root of unity $\zeta$. Let $E\coloneqq F(x,y)$, where $x$ and $y$ are independent variables over $F$, let $X$ be the Severi-Brauer variety of the degree-$p$ cyclic algebra $(x,y)$ over $E$, and let $L\coloneqq E(X)$. Consider the following elements of $E^\times$: 
		\[a\coloneqq 1-x,\quad b\coloneqq x,\quad c\coloneqq y,\quad d\coloneqq 1-y.\]
		We have $(a,b)=(c,d)=0$ in $\Br(E)$ by the Steinberg relations \cite[Chapter XIV, Proposition 4(iv)]{serre1979local}, and hence $(a,b)=(b,c)=0$ in $\on{Br}(L)$. Moreover, $(b,c)\neq 0$ in $\Br(E)$ because the residue of $(b,c)$ along $x=0$ is non-zero, while $(b,c)=0$ in $\Br(L)$ by \cite[Theorem 5.4.1]{gille2017central}. Thus $(a,b)=(b,c)=(c,d)=0$ in $\Br(L)$.
		
		Consider the sequence of tori (\ref{t-p-s-2}) over the ground field $E$, associated to the scalars $a,b,c,d\in E^\times$ chosen above:
		\[1\to T\to P\to S\to 1.\]
		Let $E_w\subset P_L$ be the $T_L$-torsor given by the equation $N_a(u)N_d(v)=w^p$. By \Cref{keyprop}(2), the torsor $E_w$ is non-trivial over $L$. Now \Cref{generic-var} implies that the Massey product
		$\langle a,b,c,d\rangle$ is not defined over $L$. In particular, by \Cref{formal-defined}, the differential graded ring $C^{\smallbullet}(\Gamma_L,\Z/p\Z)$ is not formal.
	\end{proof}

	\appendix
	
	\section{Homological algebra}\label{appendix-a}
	Let $G$ be a profinite group, and let \begin{equation}\label{4-term-a}
		0\to A_0\xrightarrow{\alpha_0} A_1\xrightarrow{\alpha_1} A_2\xrightarrow{\alpha_2} A_3\to 0\end{equation}
	be an exact sequence of discrete $G$-modules. We break (\ref{4-term-a}) into two short exact sequences
	\[0\to A_0\xrightarrow{\alpha_0} A_1\to A\to 0,\]
	\[0\to A\to A_2\xrightarrow{\alpha_2} A_3\to 0.\]
	
	We obtain a homomorphism
	\begin{equation}\label{map-phi}\theta\colon \on{Ker}[H^1(G,A_1)\xrightarrow{\alpha_1} H^1(G,A_2)]\to \on{Coker}[A_2^G\xrightarrow{\alpha_2} A_3^G]\end{equation}
	defined as the composition of the map
	\[\on{Ker}[H^1(G,A_1)\xrightarrow{\alpha_1} H^1(G,A_2)]\to \on{Ker}[H^1(G,A)\to H^1(G,A_2)]\]
	and the inverse of the isomorphism
	\begin{equation}\label{map-phi-two}\on{Coker}[A_2^G\xrightarrow{\alpha_2} A_3^G]\xrightarrow{\sim} \on{Ker}[H^1(G,A)\to H^1(G,A_2)]
	\end{equation}
	induced by the connecting homomorphism $A_3^G\to H^1(G,A)$. 
	
	\begin{lemma}\label{theta-exact}
		We have an exact sequence
		\begin{align*}
			H^1(G,A_0)\xrightarrow{\alpha_0} \on{Ker}[H^1(G,A_1)\xrightarrow{\alpha_1}H^1(G,A_2)]\xrightarrow{\theta} \on{Coker}[A_2^G\to A_3^G]\to H^2(G,A_0),
		\end{align*}
		where the last map is defined as the composition of (\ref{map-phi-two}) and the connecting homomorphism $H^1(G,A)\to H^2(G,A_0)$.
	\end{lemma}
	
	\begin{proof}
		The proof follows from the definition of $\theta$ and the exactness of (\ref{4-term-a}).
	\end{proof}

	Consider a commutative diagram of discrete $G$-modules
	\begin{equation}\label{a-b-c}   
		\begin{tikzcd}
			A_0 \arrow[r,hook,"\alpha_0"] \arrow[d,hook,"\iota_0"]  &  A_1 \arrow[r,"\alpha_1"]  \arrow[d,hook,"\iota_1"]  & A_2 \arrow[r,->>,"\alpha_2"] \arrow[d,hook,"\iota_2"] & A_3\arrow[d,hook,"\iota_3"] \\
			B_0 \arrow[r,hook,"\beta_0"] \arrow[d,->>,"\pi_0"]  & B_1 \arrow[r,"\beta_1"] \arrow[d,->>,"\pi_1"] & B_2 \arrow[r,->>,"\beta_2"] \arrow[d,->>,"\pi_2"] & B_3 \arrow[d,->>,"\pi_3"] \\
			C_0 \arrow[r,hook,"\gamma_0"]& C_1 \arrow[r,"\gamma_1"] & C_2 \arrow[r,->>,"\gamma_2"]  & C_3
		\end{tikzcd}
	\end{equation}
	with exact rows and columns. It yields a commutative diagram of abelian groups
	\begin{equation}\label{p}
		\begin{tikzcd}
			A_1^G \arrow[r,"\alpha_1"] \arrow[d,hook,"\iota_1"]  & A_2^G \arrow[r,"\alpha_2"] \arrow[d,hook,"\iota_2"] & A_3^G \arrow[d,hook,"\iota_3"] \\
			B_1^G \arrow[r,"\beta_1"] \arrow[d,"\pi_1"]& B_2^G \arrow[r,"\beta_2"]\arrow[d,"\pi_2"] & B_3^G \arrow[d,"\pi_3"]  \\
			C_1^G \arrow[d,"\partial_1"] \arrow[r,"\gamma_1"]  & C_2^G \arrow[r,"\gamma_2"] \arrow[d,"\partial_2"] & C_3^G\\
			H^1(G,A_1) \arrow[r,"\alpha_1"] & H^1(G,A_2)
		\end{tikzcd}
	\end{equation}
	where the columns are exact and the rows are complexes. Suppose that the connecting homomorphism $\partial_1\colon C_1^G\to H^1(G,A_1)$ is surjective. We define a function
	\[\theta'\colon \on{Ker}[H^1(G,A_1)\xrightarrow{\alpha_1} H^1(G,A_2)]\to \on{Coker}(A_2^G\xrightarrow{\alpha_2} A_3^G)\]
	as follows.
	Let $z\in H^1(G,A_1)$ such that $\alpha_1(z)=0$ in $H^1(G,A_2)$. By assumption, there exists $c_1\in C_1^G$ such that $\partial_1(c_1)=z$. By the exactness of the second column, there exists $b_2\in B_2^G$ such that $\pi_2(b_2)=\gamma_1(c_1)$. By the exactness of the first column and the injectivity of $\iota_3$, there exists a unique element $a_3\in A_3^G$ such that $\beta_2(b_2)=\iota_3(a_3)$. We set \[\theta'(z)\coloneqq a_3+\alpha_2(A_2^G).\] A diagram chase shows that $\theta'$ is a well-defined homomorphism.
	
	\begin{lemma}\label{partial=partial'}
		Let $G$ be a profinite group, and suppose given an exact sequence (\ref{4-term-a}) and a commutative diagram (\ref{a-b-c}) such that the connecting homomorphism $\partial_1\colon C_1^G\to H^1(G,A_1)$ is surjective. Then $\theta=-\theta'$. 
	\end{lemma}
	
	\begin{proof}
		Let $z\in H^1(G,A_1)$ be such that $\alpha_1(z)=0$ in $H^1(G,A_2)$. Since the map $\partial_1\colon C_1^G\to H^1(G,A_1)$ is surjective, there exists $c_1\in C_1^G$ such that $\partial_1(c_1)=z$. Let $b_1\in B_1$ be such that $\pi_1(b_1)=c_1$, and for all $g\in G$ let $a_{1g}$ be the unique element of $A_1$ such that $\iota(a_{1g})=gb-b$. Then $\partial_1(c_1)$ is represented by the $1$-cocycle $\set{a_{1g}}_{g\in G}$. 
		
		Define $b_2\coloneqq \beta_1(b_1)$ and $c_2\coloneqq \gamma_1(c_1)$, so that $\pi_2(b_2)=c_2$. Since $\alpha_1(z)=0$ is represented by the cocycle $\set{\alpha_1(a_{1g})}_{g\in G}$, we deduce that there exists $a_2\in A_2$ such that $\alpha_1(a_{1g})=ga_2-a_2$ for all $g\in G$. It follows that $gb_2-b_2=\iota_2(ga_2-a_2)$ for all $g\in G$, that is, $b_2-\iota_2(a_2)$ belongs to $B_2^G$. Moreover, we have \[\pi_2(b_2-\iota_2(a_2))=\pi_2(b_2)=\gamma_1(c_1).\]
		
		Finally, we have 
		\[\beta_2(b_2-\iota_2(a_2))=\beta_2(\beta_1(b_1))-\iota_3(\alpha_2(a_2))=\iota_3(-\alpha_2(a_2)).\] 
		By definition, $\theta'(z)=-\alpha_2(a_2)+\alpha_2(A_2^G)$. Note that $\alpha_2(a_2)$ belongs to $A_2^G$ because for all $g\in G$ we have
		\[g\alpha_2(a_2)-\alpha_2(a_2)=\alpha_2(ga_2-a_2)=\alpha_2(\alpha_1(a_{1g}))=0.\]
		For all $g\in G$, let $a_g\in A$ be the image of $a_{1g}$. The homomorphism
		\[\on{Ker}[H^1(G,A_1)\xrightarrow{\alpha_1} H^1(G,A_2)]\to \on{Ker}[H^1(G,A)\to H^1(G,A_2)]\]
		induced by the map $A_1\to A$ sends the class of $\set{a_{1g}}_{g\in G}$ to the class of $\set{a_g}_{g\in G}$. 
		
		The element $a_2\in A_2$ is a lift of $\alpha_2(a_2)$. As $ga_2-a_2=\alpha_1(a_{1g})$ for all $g\in G$, the injective map $A\to A_2$ sends $a_g$ to $ga_2-a_2$ for all $g\in G$. Therefore, the connecting homomorphism $A_3^G\to H^1(G,A)$ sends $\alpha_2(a_2)$ to the class of $\set{a_g}_{g\in G}$. It follows that the isomorphism
		\[\on{Coker}[A_2^G\xrightarrow{\alpha_2} A_3^G]\xrightarrow{\sim} \on{Ker}[H^1(G,A)\to H^1(G,A_2)]\]
		induced by $A_3^G\to H^1(G,A)$ sends $\alpha_2(a_2)+\alpha_2(A_2^G)$ to the class of $\set{a_g}_{g\in G}$. By the definition of $\theta$, we conclude that $\theta(z)=\alpha_2(a_2)+\alpha_2(A_2^G)=-\theta'(z)$.    
	\end{proof}
	
	\section{Unramified torsors under tori}\label{appendix-b}
	Let $F$ be a field, let $X$ be a smooth projective geometrically connected $F$-variety, let $K$ be a Galois extension of $F$ (possibly of infinite degree over $F$), and let $G\coloneqq \on{Gal}(K/F)$. We have an exact sequence of discrete $G$-modules
	\begin{equation}\label{4exact}
		1\to K^\times \to K(X)^\times\xrightarrow{\on{div}} \on{Div}(X_K)\xrightarrow{\lambda} \on{Pic}(X_K)\to 0,
	\end{equation}
	where $\on{div}$ takes a non-zero rational function $f\in K(X)^\times$ to its divisor, and $\lambda$ takes a divisor on $X_K$ to its class in $\on{Pic}(X_K)$.
	
	Let $T$ be an $F$-torus split by $K$. Write $T_*$ for the cocharacter lattice of $T$: it is a finitely generated $\Z$-free $G$-module. Tensoring (\ref{4exact}) with $T_*$, we obtain an exact sequence of $G$-modules
	\begin{equation}\label{4exact-t}
		1\to T(K) \to T(K(X))\xrightarrow{\on{div}} \on{Div}(X_K)\otimes T_*\xrightarrow{\lambda} \on{Pic}(X_K)\otimes T_*\to 0,
	\end{equation}
	where we have used the fact that $K^\times\otimes T_*=T(K)$.
	
	We define the subgroup of unramified torsors \[H^1(G,T(K(X)))_{\on{nr}}\coloneqq \on{Ker}[H^1(G,T(K(X)))\xrightarrow{\on{div}}H^1(G,\on{Div}(X_K\otimes T_*))].\]
	The sequence (\ref{4exact}) is a special case of (\ref{4-term-a}). In this case, the map $\theta$ of (\ref{4-term-a}) takes the form
	\begin{equation}\label{theta-t}\theta\colon H^1(G,T(K(X)))_{\on{nr}}\to \on{Coker}[\on{Div}(X_K)\otimes T_*\xrightarrow{\lambda}\on{Pic}(X_K)\otimes T_*].\end{equation}

	\begin{prop}\label{theta-tori-exact}
		We have an exact sequence 
		\begin{equation*}
			\resizebox{\textwidth}{!}{$
				H^1(G, T(K)) \to H^1(G, T(K(X)))_{\on{nr}}
				\xrightarrow{\theta} \on{Coker} [(\on{Div}(X_K)\otimes T_*)^G \xrightarrow{\lambda}(\on{Pic}(X_K)\otimes T_*)^G] \to H^2(G, T(K)),$}
		\end{equation*}
		where the first map and the last map are induced by (\ref{4exact-t}).
	\end{prop}
	
	\begin{proof}
		This is a special case of \Cref{theta-exact}.
	\end{proof}
	
	By \Cref{partial=partial'}, the map $\theta$ may be computed as follows. Let
	\begin{equation}\label{t-p-s}
		1\to T\xrightarrow{\iota} P\xrightarrow{\pi} S\to 1
	\end{equation}
	be a short exact sequence of $F$-tori split by $K$ such that $P$ is a quasi-trivial torus. Passing to cocharacter lattices, we obtain a short exact sequence of $G$-modules
	\begin{equation}\label{t-p-s-cocharacter}0\to T_*\xrightarrow{\iota_*} P_*\xrightarrow{\pi_*} S_*\to 0.
	\end{equation}

	We tensor (\ref{4exact}) with $T_*$, $P_*$ and $S_*$ respectively, and pass to group cohomology to obtain the following commutative diagram, where the columns are exact and the rows are complexes:
	\begin{equation}\label{big-diagram}
		\begin{tikzcd}    & \on{Div}((X_K)\otimes T_*)^G \arrow[d,hook,"\iota_*"] \arrow[r,"\lambda"] & (\on{Pic}(X_K)\otimes T_*)^G \arrow[d,hook,"\iota_*"] \\
			P(F(X)) \arrow[d,"\pi_*"] \arrow[r,"{\on{div}}"] & (\on{Div}(X_K)\otimes P_*)^G \arrow[d,"\pi_*"] \arrow[r,"\lambda"] & (\on{Pic}(X_K)\otimes P_*)^G \arrow[d,"\pi_*"] \\
			S(F(X)) \arrow[d,twoheadrightarrow,"\partial"] \arrow[r,"{\on{div}}"] & (\on{Div}(X_K)\otimes S_*)^G \arrow[d,"\partial"] \arrow[r,"\lambda"] & (\on{Pic}(X_K)\otimes S_*)^G \\
			H^1(G,T(K(X))) \arrow[r,"{\on{div}}"] & H^1(G,\on{Div}(X_K\otimes T_*)).
		\end{tikzcd}
	\end{equation}
	Note that $\on{Gal}(K(X)/F(X))=G$. Therefore $H^1(G,P(K(X)))$ is trivial, and hence $\partial\colon S(F(X))\to H^1(G,T(K(X)))$ is surjective.
	
	Let $\tau\in H^1(G, T(K(X)))_{\on{nr}}$, choose $\sigma\in S(F(X))$ such that $\partial(\sigma)=\tau$. Then pick $\rho\in (\on{Div}(X_K)\otimes P_*)^G$ such that $\pi_*(\rho)=\on{div}(\sigma)$, and let $t$ be the unique element in $(\on{Pic}(X_K)\otimes T_*)^G$ such that $\lambda(\rho)=\iota_*(t)$. \Cref{partial=partial'} implies
	\begin{equation}\label{theta'}
		\theta(\tau)=-t.
	\end{equation}
	
	Finally, suppose that $K=F_s$ is a separable closure of $F$, so that $G=\Gamma_F$, and write $X_s$ for $X\times_FF_s$. The exact sequence (\ref{4exact-t}) for $K=F_s$ takes the form
	\begin{equation}\label{4exact-t-galois}
		1\to T(F_s) \to T(F_s(X))\xrightarrow{\on{div}} \on{Div}(X_s)\otimes T_*\xrightarrow{\lambda} \on{Pic}(X_s)\otimes T_*\to 0.
	\end{equation}

	We have the inflation-restriction sequence
	\[0\to H^1(F,T(F_s(X)))\xrightarrow{\on{Inf}}H^1(F(X),T)\xrightarrow{\on{Res}}H^1(F_s(X),T).\]
	Since $T$ is defined over $F$, it is split by $F_s$, and hence by Hilbert's Theorem 90 we have $H^1(F_s(X),T)$=0. Thus the inflation map $H^1(F,T(F_s(X)))\to H^1(F(X),T)$ is an isomorphism. We identify $H^1(F,T(F_s(X)))$ with $H^1(F(X),T)$ via the inflation map. If we define \[H^1(F(X),T)_{\on{nr}}\coloneqq \on{Ker}[H^1(F(X),T)\xrightarrow{\on{div}}H^1(F,\on{Div}(X_s)\otimes T_*)],\]
	the map $\theta$ of (\ref{map-phi}) takes the form
	\[\theta\colon H^1(F(X),T)_{\on{nr}}\to \on{Coker}[\on{Div}(X_s)\otimes T_*\to\on{Pic}(X_s)\otimes T_*].\]
	
	\begin{cor}\label{theta-tori-exact-galois}
		We have an exact sequence
		\begin{equation*}
			\resizebox{\textwidth}{!}{$
				H^1(F, T) \to H^1(F(X), T)_{\on{nr}} \xrightarrow{\theta} \on{Coker} [(\on{Div}(X_s)\otimes T_*)^{\Gamma_F}
				\xrightarrow{\lambda} (\on{Pic}(X_s)\otimes T_*)^{\Gamma_F}] \to H^2(F, T),$}
		\end{equation*}
		where the first and last map are induced by (\ref{4exact-t-galois}).
	\end{cor}
	
	\begin{proof}
		This is a special case of \Cref{theta-tori-exact}.
	\end{proof}

\end{document}